\documentclass[english,reqno,11pt]{amsart}
\usepackage{amsmath, amsthm, amsfonts}
\usepackage{hyperref} 
\usepackage{hyperref}
\hypersetup{
	colorlinks=true,
		citecolor=blue!60!black,
		linkcolor=red!60!black,
		urlcolor=green!40!black,
		filecolor=yellow!50!black,
	breaklinks=true,
	pdfpagemode=UseNone,
	bookmarksopen=false,
}
\usepackage{tcolorbox}
\usepackage{amsmath,amsthm,amssymb}
\usepackage{amscd,indentfirst,epsfig}
\usepackage{latexsym}
\usepackage{times}
\usepackage{enumerate}
\usepackage{mathrsfs}
\usepackage{stmaryrd}
\usepackage{amsopn}
\usepackage{amsmath}
\usepackage{amssymb,dsfont,mathtools}
\usepackage{amsfonts,bm}
\usepackage{amsbsy,amsmath}
\usepackage{amscd}
\usepackage{xcolor}

\usepackage[utf8]{inputenc}
\usepackage[mono=false]{libertine}
\usepackage[T1]{fontenc}

\usepackage{amsthm}
\usepackage[normalem]{ulem} 
\usepackage[cal=euler, scr=boondoxo]{}

\usepackage{microtype}

\usepackage{numprint}

\begin{document}

\makeatletter
\def \le {\leqslant}
\def \ge {\geqslant}
\def \leq {\leqslant}
\def \geq {\geqslant}
\def\e{\alpha}
\def\tend{\rightarrow}
\def\R{\mathbb R}
\def\S{\mathbb S}
\def\Z{\mathbb Z}
\def\N{\mathbb N}
\def\C{\mathcal C}
\def\D{\mathcal D}
\def\T{\mathcal T}
\def\E{\mathcal E}
\def\s{\sigma}
\def\k{\kappa}
\def \a{\beta}
 \def\be{\beta}
\def\t{\theta}
\def\l{\ell}
\def \M {\mathcal{M}}
\def\L{\bm{A}_{\g}}
\def\O{\Omega}
\def\g{\gamma}
\def \ds {\displaystyle}
\def\G{\Gamma}
\def \mM {\mathfrak{m}}
\def\f{\varphi}
\def\o{\omega}
\def \d {\mathrm{d}}
\def \lm {\bm{m}}
\def \lM {\mathds{M}}
\def \lD {\mathds{D}}
\def\U{\Upsilon}
\def\z{\zeta}
\def \Q {\mathcal{Q}}
\def\over{\bm}
\def\b{\backslash}
\def\fet{f_{\ast}}
\def \get{g_{\ast}}
\def\fprim{f^{\prime}}
\def\fprimet{f^{\prime}_\ast }
\def\vet{v_{\ast}}
\def  \pa {\partial}
\def \var {\dd}
\def\vprim{v^{\prime}}
\def\vprimet{v^{\prime}_{\ast}}
\def\grad{\nabla}
\def\Log{\textrm{Log }}
\newtheorem{theo}{Theorem}[section]
\newtheorem{prop}[theo]{Proposition}
\newtheorem{cor}[theo]{Corollary}
\newtheorem{lem}[theo]{Lemma}
\newtheorem{hyp}[theo]{Assumptions}
\newtheorem{defi}[theo]{Definition}
\newtheorem{rmq}[theo]{Remark}
\DeclarePairedDelimiter\lnorm{\llbracket}{\rrbracket}		

\def \dd {\bm{\varepsilon}}

\def \ind {\mathbf{1}}
\numberwithin{equation}{section}

\title[ Landau-Fermi-Dirac equation]{About the use of entropy production for the Landau-Fermi-Dirac equation}

\author{R. Alonso}
\address{$^1$ Texas A\&M University at Qatar, Science Department, Education City, Doha, Qatar. \newline
\indent $^2$ Departamento de Matem\'atica, PUC-Rio, Rio de Janeiro, Brasil.} \email{ricardo.alonso@qatar.tamu.edu}

\author{V. Bagland}
\address{Universit\'{e} Clermont Auvergne, LMBP, UMR 6620 - CNRS,  Campus des C\'ezeaux, 3, place Vasarely, TSA 60026, CS 60026, F-63178 Aubi\`ere Cedex, France.}\email{Veronique.Bagland@math.univ-bpclermont.fr}

\author{L. Desvillettes}
\address{Universit\'{e} de Paris and Sorbonne Universit\'e, CNRS, Institut de Math\'{e}matiques de Jussieu-Paris Rive Gauche,  F-75013, Paris, France} \email{desvillettes@math.univ-paris-diderot.fr}

\author{B.  Lods}
\address{Universit\`{a} degli
Studi di Torino \& Collegio Carlo Alberto, Department of Economics, Social Sciences, Applied Mathematics and Statistics ``ESOMAS'', Corso Unione Sovietica, 218/bis, 10134 Torino, Italy.}\email{bertrand.lods@unito.it}

\maketitle

\begin{abstract}
In this paper, we present new estimates for the entropy dissipation of the Landau-Fermi-Dirac equation (with hard or moderately soft potentials) in terms of a weighted relative Fisher information adapted to this equation.  Such estimates are used for studying the large time behaviour of the equation, as well as for providing new a priori estimates (in the soft potential case).  An important feature of such estimates is that they are uniform with respect to the quantum parameter.   Consequently, the same estimations are recovered for the classical limit, that is the Landau equation.

 
\end{abstract}

\tableofcontents

\section{Introduction}

Entropy production estimates have been significantly used in recent years for the study of partial differential equations.  Such estimates arise in many contexts for different purposes: 
\begin{itemize}
\item Entropy estimates contribute to new \emph{a priori estimates}. This is the case in the kinetic theory of gases where the famous $H$-Theorem of Boltzmann is the cornerstone in many studies on the topic and one of the essential bricks in the construction of \textsc{Di Perna-Lions} renormalized solutions \cite{DPL}. That entropy yields new natural \emph{a priori} estimates is also well-understood in the study of other equations,
 and the Boltzmann entropy $\int f\log f$ appeared in this context in the fundamental work \cite{nash} on parabolic equations. We refer the reader to \cite{jungel} for an overview of entropy methods for diffusive equations.
\item Entropy production estimates also arise naturally in the study of long time behaviour of solutions to evolution problems. Indeed, entropy acts in this case as a natural \emph{Lyapunov functional} which brings the system towards its equilibrium state (which is a minimizer of the entropy) through some LaSalle's invariance principle. In this context, functional inequalities linking the entropy to the entropy production are a fundamental tool to \emph{quantify} the rate of convergence. In the study of kinetic theory, this has been understood since the celebrated \emph{Cercignani's conjecture} for Boltzmann equation \cite{cerc, vilcerc, DMouVil} and the pioneering works \cite{carlen,carlen1}. Such ideas have been applied efficiently also in various other frameworks, for the study of reaction-diffusion systems \cite{fellner}, parabolic equations \cite{To}, coagulation and fragmentation processes \cite{aizenman,canizo} and in various problems in mathematical biology \cite{perthame, michel}, just to mention a few. 
\end{itemize}

In the kinetic framework, entropy and entropy production have been thoroughly studied for the Boltzmann equation \cite{vilcerc,DeVi3} as well as for the Landau equation. This latter equation arises in the modelling of plasma and can also be derived from the Boltzmann equation in the so-called grazing collision limit. For this model, the equivalent of Cercignani's conjecture was proven to hold first in the  special case of the so-called Maxwell molecules in \cite{DeVi2}, together with weaker versions of the functional inequality linking the entropy to its entropy production for hard potentials. This kind of functional inequalities has been then extended to cover the  physically relevant case of Coulomb interactions in \cite{CDH} using techniques from \cite{DesvJFA}. 
\medskip

In a recent contribution \cite{ABL}, the entropy method has been applied to the study of the long time behaviour of solutions to the Landau-Fermi-Dirac equation with \emph{hard potentials}. Some non optimal functional inequalities linking the entropy production to the associated entropy have been obtained therein which, when combined with a careful spectral analysis, yield an optimal (exponential) rate of convergence to equilibrium. It is the purpose of this paper to provide a systematic study of the entropy and entropy production functional for the spatially homogeneous Landau-Fermi-Dirac equation, extending  the results obtained in \cite{ABL}  and complementing them with a study of the \emph{soft potentials} case. The results of the present contribution in that context will then be applied to the long-time behaviour of solutions to the Landau-Fermi-Dirac equation with moderately soft potentials in the forthcoming work \cite{ABDL-soft}.

\subsection{The model} 

The Landau-Fermi-Dirac (LFD) equation models the time evolution of a particle gas satisfying Pauli's exclusion principle in the Landau's grazing limit regime and reads, in the homogeneous setting, as
\begin{equation}\label{LFDeq}
\partial_{t} f(t,v) =  \Q_{\dd}(f)(t,v),\qquad (t,v)\in (0,\infty)\times\mathbb{R}^{3}\,, \qquad f(0)=f_{\mathrm{in}}\,,
\end{equation}
where the collision operator is given by a modification of the Landau operator which includes Pauli's exclusion principle:
\begin{equation*}
\Q_{\dd}(f)(v)= {\grad}_v \cdot \int_{\R^3} \Psi(v-\vet) \, \Pi(v-\vet) 
\Big\{ \fet (1- \dd \fet) \grad f - f (1- \dd f) {\grad f}_\ast \Big\}
\, \d\vet\,.
\end{equation*}
We use here the standard shorthand $f=f(t,v)$ and $\fet=f(t,\vet)$.  The matrix-valued function $z \mapsto \Pi(z)$ denotes the 
orthogonal projection on $(\R z)^\perp$, 
$$\Pi_{ij}(z)=\delta_{ij}-\frac{z_i z_j}{|z|^2}, \qquad 1\leq i , j\leq 3\,,$$
and  $\Psi(z)=|z|^{2+\g}$ is the kinetic potential.  The choice $\Psi(z)=|z|^{2+\g}$ corresponds 
to inverse power law potentials.    We point out that the Pauli exclusion 
principle implies that a solution to \eqref{LFDeq} must satisfy the \emph{a priori} bound 
\begin{equation}\label{eq:Linf}
0\leq f(t,v)\leq \dd^{-1},
\end{equation}
where the \textit{quantum parameter} 
$$\dd:= \frac{(2\pi\hslash)^{3}}{m^{3}\beta} >0$$ 
depends on the reduced Planck constant $\hslash \approx 1.054\times10^{-34} \mathrm{m}^{2}\mathrm{kg\,s}^{-1}$, the mass $m$, and the statistical weight $\beta$ of the particles species, see \cite[Chapter 17]{chapman}.  Recall that the statistical weight is the number of independent quantum states in which the particle can have the same internal energy.  For example, for electrons, $\beta=2$ corresponding to the two possible electron spin values.  In the case of electrons $m\approx 9.1\times10^{-31}$ kg, and therefore, $\dd \approx 1.93\times10^{-10}\ll 1$. The parameter $\dd$ encapsulates the quantum effects of the model,  the case $\dd=0$ corresponding to the classical Landau equation as studied in \cite{DeVi2}. 

The above equation shares many properties with the classical Landau equation, corresponding to non quantum particles:
\begin{equation}\label{eq:landau}
\partial_{t}f=\Q_{0}(f)={\grad}_v \cdot \int_{\R^3} \Psi(v-\vet) \, \Pi(v-\vet) 
\Big\{ \fet \grad f - f {\grad f}_\ast \Big\}
\, \d\vet\ ,
\end{equation}
which corresponds to the case $\dd=0.$ As mentioned earlier, Eq. \eqref{eq:landau} is a fundamental model of kinetic theory for plasmas and received considerable attention in the past decades, cf. for example \cite{DeVi,DeVi2,CDH,kleb}.

{Up to our knowledge, the mathematical study of \eqref{LFDeq} has been restricted to the case of \emph{hard potentials} (or Maxwell molecules) only, i.e. for $\gamma \in [0,1]$.} 
We refer to \cite{ABL,bag1} for a  discussion of the model, its physical relevance and the properties of the solutions {for $\g \in [0,1]$}. We point out here that the study of the Cauchy problem for \eqref{LFDeq} has been performed in \cite{bag1}, and the careful study of the steady states is given in \cite{bag2}. In a recent contribution \cite{ABL}, three of the authors of the present paper discussed both the regularity and the long time asymptotics of the solution to \eqref{LFDeq}. {The analysis of \eqref{LFDeq} for moderately soft potentials, corresponding to $-2 < \gamma <0$, has been initiated by the authors of the present paper in \cite{ABDL-soft}}.

\subsection{The role of entropy} In the classical context (corresponding to $\dd=0$), the Boltzmann entropy
$$H(f)=\int_{\R^{3}}f\log f\d v$$
is a Lyapunov functional for \eqref{eq:landau}, i.e. 
$$\dfrac{\d}{\d t}H(f(t)) \leq 0 \qquad {\hbox{ for }} \qquad t \geq0, $$
if $f(t)$ is a suitable solution to \eqref{eq:landau}. Moreover,  given any constant $\bm{u} \in \R^{3}$, under the constraints
\begin{equation}\label{eq:Gibbs}
\int_{\R^{3}}f(v) \d v = \varrho, \qquad \int_{\R^{3}}f(v)\,|v - \bm{u}|^{2}\d v\leq 3\varrho\,E, \qquad f \geq 0,\end{equation}
the functional $H(f)$ reaches its unique minimum if $f$
is the Maxwellian distribution
$$M(v) :=\frac{\varrho}{(2\pi\,E)^{\frac{3}{2}}}\exp\left(-\frac{|v-\bm{u}|^{2}}{2E}\right), \qquad v \in \R^{3}, $$
for which the above constraints are satisfied with equality sign. Of course, such a Maxwellian distribution is the only solution to 
$$\Q_{0}(M)=0,$$
satisfying \eqref{eq:Gibbs} with equality sign. 

In the quantum case, for $\dd >0$, one introduces the Fermi-Dirac entropy:
\begin{equation}\label{eq:FDentro}
\mathcal{S}_{\dd}(f)= - \frac{1}{\dd}\int_{\R^3} \Big[\dd f\log(\dd f)+(1-\dd f)\log (1-\dd f)\Big] \, \d v\,,
\end{equation}
i.e.   $\mathcal{S}_{\dd}(f)= -\dd^{-1}\left(H(\dd\,f)+H(1-\dd f)\right)$, 
well-defined for any $0 \leq f\leq \dd^{-1}.$ One can then show that $-\mathcal{S}_{\dd}(f)$ is a Lyapunov function for \eqref{LFDeq}, i.e.
$$-\dfrac{\d}{\d t}\mathcal{S}_{\dd}(f(t))=:-\mathscr{D}_{\dd}(f(t)) \leq 0$$
for any suitable solution to \eqref{LFDeq}. For $\dd \geq 0$, the entropy production associated to the above Landau-Fermi-Dirac operator is defined as
\begin{equation}\label{eq:product}
\mathscr{D}_{\dd}(g):=\frac{1}{2}\int_{\R^{3}\times\R^{3}}\Psi(v-\vet)\bm{\Xi}_{\dd}[g](v,\vet)\d v\d\vet\,,\qquad \Psi(z)=|z|^{\g+2}\,,
\end{equation}
for any smooth function $0 < g < \dd^{-1}$, with 
\begin{equation}\label{eq:Xidd}\begin{split}
\bm{\Xi}_{\dd}[g](v,\vet)&=\Pi(v-\vet)\big(g_{\ast}(1-\dd g_{\ast})\nabla g - g(1-\dd g)\nabla g_{\ast}\big)\left(\frac{\nabla g}{g(1-\dd g)}-\frac{\nabla g_{\ast}}{g_{\ast}(1-\dd g_{\ast})}\right)\\
&=gg_{\ast}(1-\dd g)(1-\dd g_{\ast})\left|\Pi(v-\vet)\left(\frac{\nabla g}{g(1-\dd g)}-\frac{\nabla g_{\ast}}{g_{\ast}(1-\dd g_{\ast})}\right)\right|^{2} \geq 0\,.
\end{split}
\end{equation}
In contrast to what happens in the classical case, here the steady solutions to \eqref{LFDeq} satisfying \eqref{eq:Gibbs} are of two different kinds: first, $\Q_{\dd}(\M_{\dd})=0$ if $\M_{\dd}$ is the following Fermi-Dirac statistics:
\begin{equation}\label{eq:FDS}
\M_{\dd}(v) :=\frac{a_{\dd}\exp(-b_{\dd}|v-\bm{u}|^{2})}{1+\dd\,a_{\dd}\exp(-b_{\dd}|v-\bm{u}|^{2})}=\frac{M_{\dd}}{1+\dd\,M_{\dd}},
\end{equation}
where the parameters $a_{\dd},b_{\dd}$ are semi-explicit and tuned so that $\M_{\dd}$ satisfies \eqref{eq:Gibbs} with equality sign. Second, the distribution
\begin{equation}\label{eq:dege}
F_{\dd}(v)=\dd^{-1}\ind_{B(\bm{u},R)}(v),
\end{equation}   
where $B(\bm{u},R)$ is the open ball in $\R^{3}$ centered at $\bm{u}$ with explicit radius $R:= \left(\dfrac{3\varrho\dd}{|\S^{2}|} \right)^{\frac{1}{3}}$,
also satisfies $\Q_{\dd}(F_{\dd})=0$ and \eqref{eq:Gibbs} (for $\dd$ small enough), where $|\S^{2}|$ is the volume of the unit sphere.  Such a degenerate stationary state, referred to as \emph{saturated Fermi-Dirac distribution},   can occur for very cold gases, where an explicit condition on the gas temperature can be found. Notice that, under the additional constraints:
$$0 \leq f \leq \dd^{-1},\qquad \left|\{v \in \R^{3}\;;\;0 < f(v) < \dd^{-1}\}\right| \neq 0, $$
the Fermi-Dirac statistics is the \emph{unique} stationary state (see \cite{bag2}). It is also a minimizer of the functional $-\mathcal{S}_{\dd}(f)$.

The coexistence of these two kinds of steady solutions has important consequences on the long time behaviour of solutions to \eqref{LFDeq}. Indeed, as long as finite energy solutions are not saturated, only the Fermi-Dirac statistics \eqref{eq:FDS} are possible limit points for the solution to \eqref{LFDeq}; yet, if the initial datum happens to be very close to the degenerate state \eqref{eq:dege}, one expects the convergence to be slowed down by such equilibrium regimes.  To illustrate this, let us summarize some of the results of our contribution \cite{ABL} dealing with $\g \in (0,1]$.
\begin{theo}\label{theo:conve} Consider {$0 < \g \leq 1$ and} $0\leq f_0\in L^{1}_{s_\gamma}(\R^{3})$  with $s_{\g}=\max\{\tfrac{3\gamma}{2}+2,4-\gamma\}$ satisfying \eqref{hypci}, and let $f=f(t,v)$ be a weak solution to \eqref{LFDeq} {as constructed in \cite{bag1}}. Then, 
\begin{enumerate}
\item $\lim_{t\to\infty}\|f(t)-\M_{\dd}\|_{L^{1}_{2}}=0$
where $\M_{\dd}$ is the Fermi-Dirac statistics with same mass, energy and momentum as $f_{0}$.
\item If moreover $f \in L^{2}_{k}(\R^{3})$ for $k$ large enough (explicit), then there exists  $\dd^{\ddagger} \in (0,1)$ such that for any $\dd \in (0,\dd^{\ddagger})$,
\begin{equation}\label{eq:mainestimate}
\|f(t)-\M_{\dd}\|_{L^{1}_{2}} \leq C\exp(-\lambda_{\dd}t), \qquad \forall\,t \geq 0\,,\end{equation}
where $\lambda_{\dd} >0$ is explicit, and $C>0$.
\end{enumerate}\end{theo} 

Here above and in all the sequel,
\begin{multline*}
L^{1}_{s}(\R^{3}) :=\{f \in L^{1}(\R^{3})\;;\;\lm_{s}(f) < \infty\}, \\
L^{2}_{s}(\R^{3}):=\left\{f\;;\;\|f\|_{L^{2}_{s}}^{2}:=\int_{\R^{3}}|f(v)|^{2}\langle v\rangle^{2s}\d v< \infty\right\}, \qquad s \in \R\end{multline*}
with
$$\lm_{s}(f) :=\int_{\R^{3}}\langle v\rangle^{s}\,|f(v)|\d v, \qquad \langle v\rangle:=\left(1+|v|^{2}\right)^{\frac{1}{2}}, \qquad s \in \R.$$

The optimal convergence given in \eqref{eq:mainestimate} is based upon a careful spectral analysis of the linearized Landau-Fermi-Dirac operator around $\M_{\dd}$, and  $\lambda_{\dd} >0$ is its spectral gap.  The condition on the quantum parameter $\dd\in (0,\dd^{\ddagger})$ is interpreted as a \textit{non-saturation} condition of the dynamics (and not as a smallness condition).  Indeed, if we call $\dd_{\mathrm{sat}}$ the limiting quantum parameter that saturates the initial data, the results in \cite{ABL} are valid in the range $0<\dd^{\ddagger}<c\,\dd_{\mathrm{sat}}$ for some explicit universal $c<1$.  Although convergence to Fermi-Dirac statistics still happens up to $\dd^{\ddagger}<\dd_{\mathrm{sat}}$, it is an open problem to prove such convergence with explicit rates.
  
The main scope of the present paper is to provide a direct approach to the trend to equilibrium for the evolution of the Landau-Fermi-Dirac dynamic where the convergence is obtained \textit{in terms of the relative entropy only}.  This requires sharp functional estimates.

\smallskip
\noindent
Let us introduce the \emph{relative Fermi-Dirac entropy} defined, for any nonnegative $f,\,g \in L^1_2(\R^3)$ with $0 \leq f \leq \dd^{-1}$ and  $0 \leq g \leq \dd^{-1}$, by
$$\mathcal{H}_{\dd}(f|g)=-\mathcal{S}_{\dd}(f)+\mathcal{S}_{\dd}(g).$$
The definition is similar to the usual \emph{Boltzmann relative entropy}
$$\mathcal{H}_{0}(f|g)=H(f)-H(g)=\int_{\R^{3}}f\log f \d v -\int_{\R^{3}}g\log g\, \d v ,$$
defined for nonnegative $f,g \in L^{1}(\R^{3})$ with $\int_{\R^{3}}f(v)\d v=\int_{\R^{3}}g(v)\d v$, and used in particular in the study of \eqref{eq:landau}. 
\subsection{Main results} Before stating the main results of the present contribution, we introduce the baseline class of functions to which the initial distribution $f_{\mathrm{in}}$ is associated to:
\begin{defi} \label{defi:fin}
Fix $\dd_{0}>0$ and a nonnegative $f_{\mathrm{in}}\in L^{1}_{2}(\R^{3})$ satisfying 
\begin{equation}\label{hypci}
0<\|f_{\mathrm{in}}\|_{\infty} =:\dd_{0}^{-1} <\infty\qquad \text{ and }  \qquad S_{0}:=\mathcal{S}_{\dd_{0}}(f_{\mathrm{in}})>0, \qquad |H(f_{\mathrm{in}})| < \infty .
\end{equation}
For any $\dd \in [0,\dd_{0}]$, we say that $f \in \mathcal{Y}_{\dd}(f_{\mathrm{in}})$ if $f\in L^{1}_{2}(\R^{3})$ satisfies $0\leq f\leq \dd^{-1}$,
  \begin{equation}\label{eq0}
\int_{\R^{3}}f(v)\left(\begin{array}{c}1 \\v \\ |v|^{2}\end{array}\right)\d v=\int_{\R^{3}}f_{\mathrm{in}}(v)\left(\begin{array}{c}1 \\v \\ |v|^{2}\end{array}\right)\d v=:\left(\begin{array}{c}\varrho_{\mathrm{in}} \\ \varrho_{\mathrm{in}}u_{\mathrm{in}} \\ 3\varrho_{\mathrm{in}}E_{\mathrm{in}}+\varrho_{\mathrm{in}}|u_{\mathrm{in}}|^2\end{array}\right),\end{equation}
   and $\mathcal{S}_{\dd}(f) \geq \mathcal{S}_{\dd}(f_{\mathrm{in}}).$
   \end{defi}

 {\begin{rmq}\label{rmq:Y0} The definition above can be naturally extended to cover the case $\dd=0$. Of course for this case we also assume $\dd_{0}=0$ and replace \eqref{hypci} simply with
$$f_{\mathrm{in}} \in L^{1}_{2}(\R^{3}), \qquad  \left|H(f_{\mathrm{in}})\right| < \infty ,$$
and we say that $f \in \mathcal{Y}_{0}(f_{\mathrm{in}})$ if $f \in L^{1}_{2}(\R^{3})$ is nonnegative and satisfies \eqref{eq0} together with $H(f) \le H(f_{\mathrm{in}})$. \end{rmq}
Up to replacing $f$ with 
$$\widetilde{f}(v)=\dfrac{E_{\mathrm{in}}^{\frac{3}{2}}}{\varrho_{\mathrm{in}}}f\left(\sqrt{E_{\mathrm{in}}}v+u_{\mathrm{in}}\right), \qquad v \in \R^{3}$$
  there is no loss in generality in assuming 
  \begin{equation}\label{eq:Mass}
\varrho_{\mathrm{in}}=E_{\mathrm{in}}=1, \qquad u_{\mathrm{in}}=0.\end{equation}
This assumption will be made in \emph{all the sequel} and $\M_{\dd}$ will always denote the Fermi-Dirac statistics corresponding to this normalization.
Solutions to the Landau-Fermi-Dirac (also called LFD) equation constructed in \cite{bag1} are known to belong (for all $t \ge 0$) to the class $\mathcal{Y}_{\dd}(f_{\mathrm{in}})$ if the initial datum $f_{\mathrm{in}}$ satisfies conditions \eqref{hypci}-\eqref{eq0}. This is the reason why we focus on the study of the entropy production functional $\mathscr{D}_{\dd}(g)$ for functions $g$ belonging to such a class.  More specifically, $f_{\mathrm{in}}$ will always be assumed to satisfy \eqref{hypci}-\eqref{eq0}.

The analysis developed in this manuscript allows us to revisit the contribution of \cite{ABL}, devoted to hard potentials, and to improve it in several aspects. More precisely, we provide a new entropy-entropy production estimate based on works of the third author which allows to strengthen and simplify several results given in \cite{ABL}. Namely, we prove the following inequality.
\begin{theo}\label{theo:main-entro} Assume that $g$ belongs to the class $\mathcal{Y}_{\dd}(f_{\mathrm{in}})$ together with the normalization  \eqref{eq:Mass} and assume moreover that
there exists $\kappa_{0} >0$ such that
\begin{equation}\label{infim}
\inf_{v \in \R^{3}}\left(1-\dd g(v)\right) \geq \kappa_{0}.\end{equation}
Then for $\g \geq 0$, 
$$\mathscr{D}_{\dd}(g) \geq 2\lambda(g)\left[b_{\dd}-\frac{12\dd^{2}}{\kappa_{0}^{4}}\max(\|g\|_{\infty}^{2},\|\M_{\dd}\|_{\infty}^{2})\right]
\mathcal{H}_{\dd}(g|\M_{\dd}), $$
where $\lambda(g) >0$ is given by 
{\begin{equation}\label{eq:lambdag}
\frac{1}{\lambda(g)}:= {510}\frac{\bm{e}_{\g}^{3}\,}{\kappa_{0}^{2}}\,\max(1,\bm{B}_{\g})\,\max\left(1,\lm_{2+\gamma}(g)\right)\mathscr{I}_{\gamma}(g)\,,
\end{equation}}
with
$$\mathscr{I}_{\gamma}(g)=\sup_{v \in \R^{3}}\langle v\rangle^{\gamma}\int_{\R^{3}}g(w)|w-v|^{-\gamma}\langle w\rangle^{2}\d w ,$$
and
$$\frac{1}{\bm{B}_{\g}}:=\min_{i\neq j}\inf_{\sigma \in \S^{1}}\int_{\R^{3}}\left|\sigma_{1}\frac{v_{i}}{\langle v\rangle}-\sigma_{2}\frac{v_{j}}{\langle v\rangle}\right|^{2}g(v)\d v, \qquad \frac{1}{\bm{e}_{\g}} :=\min_{i}\frac{1}{3}\int_{\R^{3}} g(v)\,v_{i}^2\,\d v\,. $$
Recall that $\M_{\dd}$ and $b_{\dd}$ were introduced in \eqref{eq:FDS}.
\end{theo}

The condition \eqref{infim} is needed here to rule out the possibility that $g$ be too close to the degenerate state \eqref{eq:dege}. We already point out that solutions $f(t,v)$ to LFD equation \eqref{LFDeq} satisfy \eqref{infim} as soon as $t>0$. We refer to \cite{ABDL-soft,ABL} for the proof of this fact {in the range $\gamma\in(-2,0]$ and $\g \in (0,1]$ respectively}, under some conditions on $\dd<\dd_{\mathrm{sat}}$.   In fact, $\kappa_0$ is uniform w.r.t. time when $t \ge t_0>0$.

It is noteworthy to observe that Theorem \ref{theo:main-entro} is valid in the classical case $\dd=0$ {with  $g \in \mathcal{Y}_{0}(f_{\mathrm{in}})$ (as described in Remark \ref{rmq:Y0}). In this case, assumption \eqref{infim} is clearly satisfied and $b_{0}=\frac{1}{2}.$ This yields the following novel version of the results in \cite{DeVi2}, which can be seen as a variant of Cercignani's conjecture for the Landau equation with true hard potentials (this result could alternatively be obtained using {the logarithmic Sobolev inequality of Gross \cite[Chapter 2]{jungel}}  and Proposition 4 extracted from \cite{LDProcX}):
\begin{cor}\label{cor:LandauFI}
Assume that $g \in \mathcal{Y}_{0}(f_{\mathrm{in}})$ together with the normalization  \eqref{eq:Mass}. The Landau entropy production 
$$\mathscr{D}_{0}(g)=\frac{1}{2}\int_{\R^{6}}|v-\vet|^{\g+2}\,gg_{\ast}\left|\Pi(v-\vet)\left(\nabla \log g-\nabla \log g_{\ast}\right)\right|^{2}\d v\d \vet$$
satisfies, for $\gamma \geq0$,
$$\mathscr{D}_{0}(g) \geq \lambda(g)\,\int_{\R^{3}}g(v)\log \frac{g(v)}{\M_{0}(v)}\d v, \qquad \M_{0}(v)={\left(2\pi\right)}^{-\frac{3}{2}}\exp\left(-\frac{|v|^{2}}{2}\right), $$
where $\lambda(g) >0$ is given by \eqref{eq:lambdag}.
\end{cor}
{For suitable solutions to Landau equation one deduces then the following corollary, which improves the entropic convergence obtained in \cite{DeVi2}:
\begin{cor} Let $\gamma \in [0,1]$. Consider $0\leq f_{\mathrm{in}}\in L^{1}_{2+\delta}(\R^{3})$, with $\delta>0$, satisfying \eqref{hypci} in the case $\dd=0=\dd_{0}$ together with the normalization \eqref{eq0}--\eqref{eq:Mass}.  Let $f=f(t,v)$ be a weak solution to the Landau equation \eqref{eq:landau} as constructed in \cite{DeVi}. Then, for any $t_{0} >0$, there exists some explicit $\mu >0$ such that
$$\mathcal{H}_{0}(f(t)|\M) \leq \mathcal{H}_{0}(f_{\mathrm{in}}|\M)\exp\left(-\mu(t-t_{0})\right), \qquad \forall t >t_{0}.$$ 
\end{cor}
\begin{proof} We give here  a sketch of the proof which illustrates the interest of deriving functional inequalities like those of Theorem \ref{theo:main-entro}. If $f(t,\cdot)$ is a solution to \eqref{eq:landau} as constructed in \cite{DeVi}, then 
$$\dfrac{\d}{\d t}\mathcal{H}_{0}(f(t)|\M)=-\mathscr{D}_{0}(f(t)), \qquad \forall t >0$$
so that, according to Corollary \ref{cor:LandauFI}, we deduce that
$$\dfrac{\d}{\d t}\mathcal{H}_{0}(f(t)|\M) \leq -\lambda(f(t))\mathcal{H}_{0}(f(t)|\M), \qquad \forall t >0.$$
One can then prove easily (see Section \ref{sec:LT} for details) that, for any $t_{0} >0$, 
$$\inf_{t \geq t_{0}}\lambda(f(t)) \geq \mu >0,$$
from which the conclusion follows by a simple application of Gronwall Lemma and the fact that $\mathcal{H}_{0}(f(t_{0})|\M) \leq \mathcal{H}_{0}(f_{\mathrm{in}}|\M)$. \end{proof}
The above line of reasoning can be still implemented for the LFD equation and an important consequence of Theorem \ref{theo:main-entro} is the following improvement of the convergence result stated in Theorem \ref{theo:conve}.}  %
\begin{theo}\label{cvex} {Let $\gamma \in (0,1]$.}
Consider $0 \leq f_{\mathrm{in}}\in L^{1}_{s_{\gamma}}(\R^{3})$, with $s_{\gamma}=\max\big\{\tfrac{3\gamma}{2}+2, 4-\gamma\big\}$, satisfying \eqref{hypci} together with the normalization \eqref{eq0}--\eqref{eq:Mass}.  Let $f=f(t,v)$ be a weak solution to the LFD equation as constructed in \cite{bag1} . Then, for any $t_{0} >0$, there exists $\dd^{\dagger} \in (0,\dd_{0})$ and $\mu >0$ (depending only on $t_{0}$ and  {$H(f_{\mathrm{in}})$}), such that
$$\mathcal{H}_{\dd}(f(t)|\M_{\dd}) \leq  \mathcal{H}_{\dd}(f_{\mathrm{in}}|\M_{\dd})\exp\left(-\mu\,(t-t_{0})\right) , \quad \qquad \forall \dd \in (0,\dd^{\dagger}), \qquad t >t_{0}.$$
\end{theo}
 It is worth noticing that the convergence here  is an \emph{entropic convergence} which, of course, implies a convergence in $L^{1}_{2}$ thanks to the following  Csiszar-Kullback inequality for Fermi-Dirac relative entropy obtained in \cite[Theorem 3]{LW}:
\begin{equation}\label{eq:csiszar}
\|f-\M_{\dd}\|_{L^{1}}^{2} \leq 2 \,\mathcal{H}_{\dd}(f|\M_{\dd}),\end{equation}
for all $0 \leq f \leq \frac{1}{\dd}$ with 
$$\int_{\R^{3}}f(v)\left(\begin{array}{c}1\\v \\|v|^{2}\end{array}\right)\d v=\left(\begin{array}{c}1\\0 \\3\end{array}\right).$$

As mentioned earlier, we aim to provide a systematic study of the entropy production which in particular applies also to the study of \emph{soft potentials}. For such soft potentials, the entropy estimate provides some fundamental \emph{a priori} estimate on solutions to \eqref{LFDeq}.
 We refer the reader to the recent contribution \cite{LDProcX} where this question is described in detail for solutions to the Landau equation \eqref{eq:landau} for Coulomb interactions. In this context, our main result can be formulated as follows:
\begin{theo} \label{cor:Fisherg} Let $0\leq f_{\mathrm{in}}\in L^{1}_{2}(\R^{3})$ be fixed, and satisfy \eqref{hypci} together with the normalization \eqref{eq0}--\eqref{eq:Mass}. Assume that $\gamma <0$ and $\dd \in (0,\dd_{0}]$. Then, there exists a positive constant $C_{0}(\gamma)$ depending only on {
$H(f_{\mathrm{in}})$} and on $\gamma$ such that 
$$\int_{\R^{3}}\left|\nabla \sqrt{g(v)}\right|^{2}\langle v\rangle^{\gamma}\d v \leq C_{0}(\gamma)\left(1+\mathscr{D}_{\dd}(g)\right), \qquad \forall 
\dd \in (0,\dd_{0}], \qquad \forall g \in \mathcal{Y}_{\dd}(f_{\mathrm{in}}).$$
\end{theo}
A similar result is known to hold for Landau equation (i.e. $\dd=0$), and is a consequence of more general functional inequalities for the entropy production $\mathscr{D}_{0}$, see \cite{DesvJFA, Desv}. In the quantum case, the situation is very similar, and Theorem \ref{cor:Fisherg} is a consequence of a more general functional inequality (see Proposition \ref{theo:func}).

\subsection{Organization of the paper} In Section \ref{sec:func}, we present the two main functional inequalities satisfied by the entropy production $\mathscr{D}_{\dd}$.  In particular the complete proof of Theorem \ref{cor:Fisherg} is presented. The method is inspired by the results of \cite{DesvJFA, LDProcX}. In Section \ref{sec:relative}, the links of the functional inequalities obtained in Section \ref{sec:func} and the relative entropy $\mathcal{H}_{\dd}(f|\M_{\dd})$ are described.  We give the full proof of Theorem \ref{theo:main-entro} which is valid for hard (and Maxwell molecules) potentials $\g  \ge 0$.  We briefly explain how such estimates can be used to the study of the LFD equation \eqref{LFDeq} for soft potentials $\g <0$, anticipating the results obtained in \cite{ABDL-soft}. We finally apply Theorem \ref{theo:main-entro} to the long time behaviour and get Theorem \ref{cvex},  extending the results of \cite{ABL}, in Section~\ref{sec:LT}.

\subsection*{Acknowledgments} R. Alonso gratefully acknowledges the support from Conselho Nacional de Desenvolvimento Cient\'ifico e Tecnol\'ogico (CNPq), grant Bolsa de Produtividade em Pesquisa (303325/2019-4).  BL gratefully acknowledges the financial support from the Italian Ministry of Education, University and Research (MIUR), ``Dipartimenti di Eccellenza'' grant 2018-2022 as well as the support  from the \textit{de Castro Statistics Initiative}, Collegio Carlo Alberto (Torino). 

\section{Two kinds of functional inequalities for entropy production}\label{sec:func}

We provide in this section the two main functional inequalities associated to the entropy production $\mathscr{D}_{\dd}(f)$.  Before doing so, we recall the following result, see \cite[Lemma 2.3 \& 2.4]{ABL}.
\begin{lem}\label{L2unif}
Let $0\leq f_{\mathrm{in}}\in L^{1}_{2}(\R^{3})$ be fixed and bounded satisfying \eqref{hypci}.  Then, for any $\dd \in (0,\dd_{0}]$, the following holds:
\begin{enumerate}
\item For any  $f \in \mathcal{Y}_{\dd}(f_{\mathrm{in}})$, it holds that
\begin{equation}\label{e0}
\inf_{0<\dd\leq \dd_{0}}\int_{|v|\leq R(f_{\mathrm{in}})} f(1-\dd f)\, \d v \geq \eta(f_{\mathrm{in}})>0\,,
\end{equation}
for some $R(f_{\mathrm{in}})>0$ and $\eta(f_{\mathrm{in}})$ depending only on  $H(f_{\mathrm{in}})$ but not on $\dd$.
\item For any $\delta >0$ there exists  {$\eta_*(\delta)>0$} depending only on   $H(f_{\mathrm{in}})$ such that for any $f \in \mathcal{Y}_{\dd}(f_{\mathrm{in}})$, and any measurable set $A\subset \R^3$, 
\begin{equation}\label{Lem6DV}
|A|\leq  {\eta_*(\delta)} \Longrightarrow \int_A f(1-\dd f)\, \d v \leq \delta.
\end{equation}
\end{enumerate}
\end{lem}

\subsection{Entropy production aiming to regularity}

We now aim to provide a control of the regularity of the solution $f(t)$ to \eqref{LFDeq} by the entropy production in the spirit of \cite{DesvJFA,Desv}.
 We first observe that we can easily adapt the computations of \cite[Theorem 2]{Desv} (which actually work also in the case $\gamma >0$) and prove the following functional inequality where, for any $s \in \R$ and any function $f=f(v)$ and $\chi=\chi(r)$ $(r \geq0$), we set
$$\mathcal{G}_{s}(\chi,f)=\int_{\R^{3}}\chi\left(\tfrac{1}{2}|v|^{2}\right)f(v)\langle v\rangle^{s}\d v . $$
\begin{prop}[\textbf{Functional inequality}]\label{theo:func} Let $g=g(v),\,M=M(v) \geq 0$ and 
$\phi = \phi(r)$ be given nonnegative functions {with $0 \leq g\leq \dd^{-1}$}. Write
\begin{equation}\label{newh}
F=g(1-\dd\,g),\qquad h=\log g-\log(1-\dd g).
\end{equation}
Then, for any $\gamma \in \R$ and any $i,j=1,2,3$, $i\neq j$,
\begin{multline*}
\Delta_{\phi, i,j}(F)^{2}\,\int_{\R^{3}}F(v)\left|\partial_{i}h(v)\right|^{2}\,M(v)\d v \\
\leq 36\,\mathcal{G}_{2}(\phi,F)^{4}\,\Bigg(\mathcal{G}_{2}(1,F\,M)\bigg[3\mathcal{G}_{1}^{2}(\phi,g)+8\mathcal{G}_{2}^{2}(|\phi'|,g)\bigg]+2\mathcal{J}_{\gamma}(\phi,M)\mathscr{D}_{\dd}(g)\Bigg),
\end{multline*}
where
$$\mathcal{J}_{\gamma}(\phi,M)=\sup_{v\in \R^{3}}M(v)\int_{\R^{3}}\phi^{2}\left(\tfrac{1}{2}|w|^{2}\right)\,F(w)\langle w\rangle^{2}|v-w|^{-\gamma}\d w , $$
and
$$\Delta_{\phi, i,j}(F)=\mathrm{Det}\left(\int_{\R^{3}}\phi\left(\tfrac{1}{2}|w|^{2}\right)F(w)\left(\begin{array}{ccc}1 & w_{i} & w_{j} \\w_{i} & w_{i}^{2} & w_{j}w_{i} \\w_{j} & w_{i}w_{j} & w_{j}^{2}\end{array}\right)\d w\right).$$
\end{prop}

One deduces then our main result in this context (Theorem \ref{cor:Fisherg}) in this way:
\begin{proof}[Proof of Theorem \ref{cor:Fisherg}] Remember that $\gamma <0$. Let us fix $\dd \in (0,\dd_{0}]$ and $g \in \mathcal{Y}_{\dd}(f_{\mathrm{in}})$. {Notice then that
$$\int_{\R^{3}}f(v)\langle v\rangle^{2}\d v=\varrho_{\mathrm{in}}\left(1+3E_{\mathrm{in}}\right)=4,$$
under assumptions \eqref{eq0}--\eqref{eq:Mass}.} We apply Proposition \ref{theo:func} with
$$M(v)=(1-\dd\,g(v))\langle v\rangle^{\gamma}, \qquad \phi(r)=(1+2r)^{\frac{\gamma}{4}}.$$
One has
$$\mathcal{G}_{2}(\phi,F)=\int_{\R^{3}}\langle v\rangle^{2+\frac{\gamma}{2}}F(v)\d v \leq  {4},$$
$$\mathcal{G}_{2}(1,F\,M)=\int_{\R^{3}}\langle v\rangle^{2+\gamma}(1-\dd g(v))^{2}g(v)\d v \leq  {4} ,$$
$$\mathcal{G}_{2}^{2}(|\phi'|,g)=\frac{\gamma^{2}}{4}\left(\int_{\R^{3}}\langle v\rangle^{\frac{\gamma}{2}}g(v)\d v\right)^{2} \leq  {\frac{\gamma^{2}}{4}}, $$
and
$$\mathcal{G}_{1}^{2}(\phi,g)=\left(\int_{\R^{3}}\langle v\rangle^{1+\frac{\gamma}{2}}g(v)\d v\right)^{2} \leq \left(\int_{\R^{3}}\langle v\rangle g(v)\d v\right)^{2} \leq  {4}$$
by Cauchy-Schwarz inequality. Moreover, there is some explicit $C_{\gamma} >0$ such that,
$$|v-w|^{-\gamma} \leq C_{\gamma}\langle v\rangle^{-\gamma}\langle w\rangle^{-\gamma}, \qquad \forall v,w \in \R^{3},$$
so that
$$\mathcal{J}_{\gamma}(\phi,M) \leq C_{\gamma}\sup_{v}(1-\dd\,g(v))\langle v\rangle^{\gamma}\,\int_{\R^{3}}\langle w\rangle^{\gamma}F(w)\langle w\rangle^{2}\langle v\rangle^{-\gamma}\langle w\rangle^{-\gamma}\d w ,$$
resulting in
$$\mathcal{J}_{\gamma}(\phi,M) \leq C_{\gamma}\int_{\R^{3}}F(w)\langle w\rangle^{2}\d w \leq  {4C_{\gamma}}.$$
Finally, 
$$\int_{\R^{3}}F(v)\left|\partial_{i}h(v)\right|^{2}\,M(v)\d v =\int_{\R^{3}}\frac{\left|\partial_{i}g(v)\right|^{2}}{g(v)}\langle v\rangle^{\gamma}\d v=4\int_{\R^{3}}\left|\partial_{i}\sqrt{g(v)}\right|^{2}\langle v\rangle^{\gamma}\d v.$$
Therefore, with such a choice of $\phi$ and $M$, we get (for all $i\neq j$)
\begin{equation}\label{eq:estimErho}
\Delta_{\phi, i,j}(F)^{2}\,\int_{\R^{3}}\left|\partial_{i}\sqrt{g(v)}\right|^{2}\langle v\rangle^{\gamma}\d v \leq 9\cdot{4^{5}\bigg(12+2\gamma^{2}+2C_{\gamma}\mathscr{D}_{\dd}(g)\bigg)}.
\end{equation}
It remains to find a lower bound for $\Delta_{\phi,i,j}(F)$. One can write
$$\Delta_{\phi, i,j}(F)=\mathrm{Det}\left(\begin{array}{ccc} \lnorm{1,1}  & \lnorm{1,w_{i}} & \lnorm{w_{j},1} \\ \lnorm{w_{i},1} & \lnorm{w_{i},w_{i}} & \lnorm{w_{i},w_{j}} \\ \lnorm{w_{j},1} & \lnorm{w_{i},w_{j}} & \lnorm{w_{j},w_{j}}\end{array}\right), $$
where $\lnorm{\cdot,\cdot}$ denotes the inner product on $L^{2}(\R^{3}, \phi(\frac12|\cdot|^2)F \d w)$. Thus, $\Delta_{\phi, i,j}(F)$ is the determinant of a Gram matrix and as such,
$$\Delta_{\phi, i,j}(F)^{\frac{1}{3}}\geq \inf_{\sigma \in \S^{2}}\int_{\R^{3}}\phi\left(\tfrac{1}{2}|w|^{2}\right)F(w)\left|\sigma_{1}+\sigma_{2}\,w_{i}+\sigma_{3}\,w_{j}\right|^{2}\d w , $$
since the right-hand side is less than any eigenvalue of the matrix. Recall that we picked $\phi(z)=(1+2z)^{\gamma/4}$. For any $\sigma=(\sigma_{1},\sigma_{2},\sigma_{3}) \in \S^{2}$, one has, for all $\tau >0, R >0$,
\begin{multline*}
\int_{\R^{3}}\phi\left(\tfrac{1}{2}|w|^{2}\right)F(w)\left|\sigma_{1}+\sigma_{2}\,w_{i}+\sigma_{3}\,w_{j}\right|^{2}\d w\\
\geq \tau^{2}\int_{B_{R}}\langle w\rangle^{\frac{\gamma}{2}}F(w)\ind_{\left\{\left|\sigma_{1}+\sigma_{2}\,w_{i}+\sigma_{3}\,w_{j}\right| \geq \tau\right\}}\d w\\
\geq \tau^{2}(1+R^{2})^{\frac{\gamma}{4}}\int_{B_{R}}F(w)\ind_{\left\{\left|\sigma_{1}+\sigma_{2}\,w_{i}+\sigma_{3}\,w_{j}\right| \geq \tau\right\}}\d w\\
=\tau^{2}(1+R^{2})^{\frac{\gamma}{4}}\,\bigg(\int_{B_{R}}F(w)\d w-\int_{B_{R}}F(w)
\ind_{\left\{\left|\sigma_{1}+\sigma_{2}\,w_{i}+\sigma_{3}\,w_{j}\right| \leq \tau\right\}}\d w\bigg),\end{multline*}
where $B_{R}=\{w \in \R^{3}\,;\,|w| \leq R\}$. Thus
\begin{multline}
\Delta_{\phi,i,j}(F)^{\frac{1}{3}} \geq \tau^{2}(1+R^{2})^{\frac{\gamma}{4}}\int_{B_{R}}g(w)(1-\dd\,g(w))\d w\\
-\tau^{2}(1+R^{2})^{\frac{\gamma}{4}}\sup_{\sigma \in \S^{2}}\int_{B_{R}}  g(w)(1-\dd\,g(w))
\ind_{\left\{\left|\sigma_{1}+\sigma_{2}\,w_{i}+\sigma_{3}\,w_{j}\right| \leq \tau\right\}}\d w.
\end{multline}
With the notations of Lemma \ref{L2unif}, let now consider $\delta >0$, to be fixed later. Arguing as in \cite[p. 141]{Desv}, we can find $\tau =\tau_{R, \delta} >0$ small enough such that
 {$$|A_{\tau_{R,\delta}}\cap B_{R}| \leq  \eta_*(\delta) \qquad \forall \sigma \in\S^{2}, $$}
where, for any $\sigma \in \S^{2}$, we introduced the set 
$A_{\tau}=\left\{\left|\sigma_{1}+\sigma_{2}\,w_{i}+\sigma_{3}\,w_{j}\right| \leq \tau\right\}$.
Then, 
 {thanks to Lemma \ref{L2unif} \textit{(2)}},
$$\Delta_{\phi,i,j}(F)^{\frac{1}{3}} \geq \tau_{R,\delta}^{2}(1+R^{2})^{\frac{\gamma}{4}}\int_{B_{R}}g(w)(1-\dd\,g(w))\d w-\tau_{R,\delta}^{2}(1+R^{2})^{\frac{\gamma}{4}}\delta.$$
Selecting  $R=R(f_{\mathrm{in}})$ given by Lemma \ref{L2unif} \textit{(1)}, and using the quantity $\eta(f_{\mathrm{in}}) > 0$ appearing in this lemma,
$$\Delta_{\phi,i,j}(F)^{\frac{1}{3}} \geq \tau_{R(f_{\mathrm{in}}),\delta}^{2}(1+R(f_{\mathrm{in}})^{2})^{\frac{\gamma}{4}}\left(\eta(f_{\mathrm{in}}) -\delta\right) . $$
We pick then  {$\delta=\frac{1}{2}\eta(f_{\mathrm{in}})$, and get that
$\Delta_{\phi,i,j}(F)$
is bounded below by some strictly positive constant}
depending only on {$H(f_{\mathrm{in}})$}. We conclude then thanks to \eqref{eq:estimErho}.\end{proof}
\begin{rmq}\label{GradEntro} {Whenever $-3 \leq \g \leq0$, for an initial datum $f_{\mathrm{in}}$ satisfying assumptions \eqref{hypci} and \eqref{eq0}, any suitable solution $f(t,\cdot)$ to Landau-Fermi-Dirac equation (as constructed for instance in \cite{ABDL-soft} in the case $-2< \g<0$) will then satisfy
$$\int_{t_{1}}^{t_{2}}\d t \int_{\R^{3}}\left|\nabla_{v}\sqrt{f(t,v)}\right|^{2}\langle v\rangle^{\gamma}\d v \leq C_{0}(\gamma)\int_{t_{1}}^{t_{2}}\left(1+\mathscr{D}_{\dd}(f(t))\right)\d t,\qquad 0 < t_{1}< t_{2}\,.$$
Since 
$$-\dfrac{\d}{\d t}\mathcal{S}_{\dd}(f(t))=-\mathscr{D}_{\dd}(f(t)) , $$
one deduces then that
$$\int_{t_{1}}^{t_{2}}\d t \int_{\R^{3}}\left|\nabla_{v}\left(\langle v\rangle^{\frac{\g}{2}}\sqrt{f(t,v)}\right)\right|^{2}\d v \leq \widetilde{C}_{0}(1+t_{2}-t_{1}), \qquad \forall  0 < t_{1} < t_{2}.$$
Notice that a different proof of such inequalities is obtained in \cite{ABDL-soft} for $-2 < \gamma <0.$}

Using the Sobolev embedding $\|u\|_{L^{6}}\lesssim \|\nabla u\|_{L^{2}}$ with $u=\langle v\rangle^{\g/2}\sqrt{f(t,\cdot)}$. {The above regularity estimate translates easily into an estimate for the weighted norm $\|\langle \cdot \rangle^{\g}{f(t,\cdot)}\|_{L^{3}}$. It is also possible to deduce (for some $p>1$) $L^{p}_{t}L^{p}_{v}$ bounds for the solution to \eqref{LFDeq} (see \cite{ABDL-soft} for details).}\end{rmq}

\subsection{Entropy production aiming to long-time behaviour}\label{sec:entLT}

We will consider in all this Section a function $f_{\mathrm{in}}$ {satisfying \eqref{hypci}--\eqref{eq0} and \eqref{eq:Mass} 
and a function $g$ belonging to the class $\mathcal{Y}_{\dd}(f_{\mathrm{in}})$, $\dd\in (0,\dd_{0}]$ (or $g \in \mathcal{Y}_{0}(f_{\mathrm{in}})$ in the case $\dd=0=\dd_{0}$, see Remark \ref{rmq:Y0}).}

Without loss of generality (it amounts to the use of a rotation of $\R^3$), we also assume that 
\begin{equation} \label{eq0bis}
 \int_{\R^{3}} g(v)\,v_i\,v_j\,\d v = 0, \qquad \quad i \neq j.
\end{equation}
Noticing that, for any $z,y \in \R^{3}$,
$$\langle |z|^{2}\Pi(z)y,y\rangle=|z|^{2}|y|^{2}-\langle z,y\rangle^{2}=\frac{1}{2}\sum_{i, j}|z_{i}y_{j}-z_{j}y_{i}|^{2}=\frac{1}{2}\sum_{i\neq j}|z_{i}y_{j}-z_{j}y_{i}|^{2} , $$
we have, with $z=v-\vet$ and $y=\nabla \varphi-\nabla \varphi_{\ast}$,
\begin{equation}\label{eq:Dqij}
\mathscr{D}_{\dd}(g)=\frac{1}{4}\sum_{ij}\int_{\R^{3}\times\R^{3}}F\,F_{\ast}\,|v-w|^{\gamma}\left|q_{ij}(v,w)\right|^{2}\d v\d w ,\end{equation}
where $F=g(1-\dd\,g)$ and
$$q_{ij}(v,w)=(v-w)_{i}\left(\partial_{j}h(v)-\partial_{j}h(w)\right)-(v-w)_{j}\left(\partial_{i}h(v)-\partial_{i}h(w)\right).$$
Notice that
$$q(v,w)=(v-w) \times \big(\nabla h(v)-\nabla h(w)\big).$$
We have clearly
\begin{multline}\label{eq:qij}
q_{ij}(v,w)=\left[v \times \nabla h(v)\right]_{ij}-v_{i}\partial_{j}h(w)+v_{j}\partial_{i}h(w)\\
-w_{i}\partial_{j}h(v)+w_{j}\partial_{i}h(v)+\big[w\times\nabla h(w)\big]_{ij}.\end{multline}
Integrating $q_{ij}(v,w)$ against $g(w)$ and using \eqref{eq0}, one has
\begin{equation} \label{eq4}
 N_{ij} :=N_{ij}(v)= \int q_{ij}(v,w)\, g(w)\,\d w=  \big[v \times \nabla h\big]_{ij}=v_{i}\partial_{j}h(v)-v_{j}\partial_{i}h(v), 
\end{equation}
while, integrating $q_{ij}(v,w)$ against $g(w)w_{i}$, one obtains
\begin{equation} \label{eq5}
 M_{ij} :=M_{ij}(v)=\int q_{ij}(v,w)\, g(w)\,w_{i}\,\d w=-a_i\, \partial_{j}h(v)  + K \,v_j - L_j,  
\end{equation}
where
\begin{multline} \label{eq6}
 a_{i} := \int_{\R^{3}} g(v)\,v_{i}^2\,\d v; \qquad K=K_{\dd}=\frac{1}{\dd}\int_{\R^{3}} \log(1 - \dd \, g(v))\, \d v; \\
 \qquad L_j := L_{j, \dd} =\frac{1}{\dd}\int_{\R^{3}} \log(1 - \dd \, g(v))\, v_j\, \d v.
\end{multline}
\begin{rmq}\label{rmq:KK} Notice that, under the additional assumption that $\inf_{v}(1-\dd g(v)) \geq \kappa_{0}$ (see \eqref{infim} hereafter), 
$$\int_{\R^{3}}|\log(1-\dd g(v))|\d v\leq-\frac{1}{\kappa_{0}}\int_{\R^{3}}(1-\dd g(v))\log(1-\dd g(v))\d v \leq \frac{\dd}{\kappa_{0}}\mathcal{S}_{\dd}(g) < \infty , $$
which makes $K$ well-defined. Moreover, $K <0$ with $\lim_{\dd\to0}K_{\dd}=-\int_{\R^{3}}g(v)\d v=-1$. One actually can check that
$$K(g)=\int_{\R^{3}}w_{i}\partial_{i}h(w)g(w)\d w, \qquad \forall i=1,2,3.$$
In particular, since $\nabla \M_{\dd}(v)=-2b_{\dd}v\,\M_{\dd}(v)(1-\dd\M_{\dd}(v))$, one has 
$$\overline{K} :=K(\M_{\dd})=-2b_{\dd}\int_{\R^{3}}\M_{\dd}(v)v_{i}^{2}\d v=-2b_{\dd} \qquad \forall i=1,2,3 , $$
since the energy of $\M_{\dd}$ is equal to $3$ and, because of radial symmetry, the directional energies all coincide. 
\end{rmq}

\begin{rmq}Finally, notice that
\begin{equation}\label{eq:6b}
\int_{\R^{3}}q_{ij}(v,w)g(w)w_{j}\d w=a_{j}\partial_{i}h(v)-Kv_{i}+L_{i}=-M_{ji}.\end{equation}
\end{rmq}

We begin with the following basic observation:
\begin{lem}\label{lem:estimMN}
Besides the above assumption, we assume that $g$ satisfies \eqref{infim} for some  $\kappa_{0} >0.$ 
Then, for any $\g \in\R$ and any $i\neq j$, 
$$\kappa_{0}^{2}\int_{\R^{3}}|M_{ij}|^{2}g(v) \langle v\rangle^{\gamma}\d v \leq 4\,\bm{I}_{\g}^{(2)}(g)\mathscr{D}_{\dd}(g), $$
and
$$\kappa_{0}^{2}\int_{\R^{3}}|N_{ij}|^{2}g(v) \langle v\rangle^{\gamma}\d v \leq 4\,\bm{I}_{\g}^{(0)}(g)\mathscr{D}_{\dd}(g), $$
where, for any $s \geq 0$,
$$\bm{I}_{\g}^{(s)}(g) :=\sup_{v\in \R^{3}}\langle v\rangle^{\gamma}\int_{\R^{3}}|v-w|^{-\gamma}g(w)|w|^{s} \d w.$$
 \end{lem}
\begin{proof} From Cauchy-Schwarz inequality, we infer that
$$|M_{ij}|^{2} \leq \left(\int_{\R^{3}}|q_{ij}(v,w)|^{2}g(w)|v-w|^{\gamma}\d w\right)\,\left(\int_{\R^{3}}|v-w|^{-\gamma}g(w)w_{i}^{2}\d w\right), $$
so that
\begin{multline*}\int_{\R^{3}}g(v)|M_{ij}|^{2}\langle v\rangle^{\gamma}\d v \leq \sup_{v\in \R^{3}}\left(\langle v\rangle^{\gamma}\int_{\R^{3}}|v-w|^{-\gamma}g(w)w_{i}^{2}\d w\right)\\
\times
\int_{\R^{3}\times\R^{3}}|q_{ij}(v,w)|^{2}|v-w|^{\gamma}g(v)g(w)\d v\d w.\end{multline*}
The last integral is bounded from above by $4\kappa_{0}^{-2}\mathscr{D}_{\dd}(g)$ according to \eqref{infim}, whereas the first integral is obviously bounded by $\bm{I}_{\g}^{(2)}(g)$, whence the result.
 The proof for $N_{ij}$ is identical.\end{proof}
\begin{rmq}\label{rmq:estimaMN} Notice that, for $\g >0$, $\bm{I}_{\g}^{(s)}(g) <\infty$ as soon as $v \mapsto |v|^{s}g(v)$ is smooth enough whereas for $\g < 0$, using that 
 {$|v-w|^{-\g} \leq \max(1, 2^{-\g - 1}) \, \langle v\rangle^{-\g}\langle w\rangle^{-\g}$ for any $v,w \in \R^{3}$, one sees that
$$\bm{I}_{\g}^{(s)} \leq   \max(1, 2^{-\g - 1}) \, \bm{m}_{s-\g}(g), \qquad s \geq0.$$}
\end{rmq} 

Notice that \eqref{eq5} can be written as
\begin{equation*}
 \partial_{j}h(v)-Kv_{j}=K\left(\frac{1}{a_{i}}-1\right)v_{j}-\frac{M_{ij}}{a_{i}}-\frac{L_{j}}{a_{i}},
\end{equation*}
so that
\begin{equation*}
\frac{1}{3}\big|\partial_{j}h(v)-Kv_{j}\big|^{2} \leq K^{2}\left|1-\frac{1}{a_{i}}\right|^{2}v_{j}^{2}+\frac{M^{2}_{ij}}{a_{i}^{2}}+\frac{L_{j}^{2}}{a_{i}^{2}},
\end{equation*}
and integrating against $g(v)\langle v\rangle^{\g}$,   we obtain 
\begin{multline}\label{eq:fishj}
\frac{1}{3}\int_{\R^{3}}\big|\partial_{j}h(v)-Kv_{j}\big|^{2} g(v)\langle v\rangle^{\gamma}\d v\\
\leq K^{2}\left|1-\frac{1}{a_{i}}\right|^{2}\int_{\R^{3}}g(v)v_{j}^{2}\langle v\rangle^{\gamma}\d v + \int_{\R^{3}}\frac{M^{2}_{ij}}{a_{i}^{2}}g(v)\langle v\rangle^{\gamma}\d v
+\frac{L_{j}^{2}}{a_{i}^{2}}\int_{\R^{3}}g(v)\langle v\rangle^{\gamma}\d v, 
\end{multline}
which holds for any $1 \leq i \neq j\leq 3$, and any $\g \in \R.$  
We need to estimate the various terms in the right-hand side of this expression.
We begin with the following:
\begin{lem} With the above notations, for any $\g \in \R$, we have
\begin{multline}\label{eq11}
K^{2}\left|1-\frac{1}{a_{i}}\right|^{2} 
\leq \frac{4}{9}\max\left(\frac{a_{j}^{2}}{A_{k,\g}}\,;\,\frac{a_{k}^{2}}{A_{j,\g}}\right) \int_{\R^{3}}\left[\frac{M^{2}_{ik}}{a_{i}^{2}}+\frac{M_{jk}^{2}}{a_{j}^{2}}+\frac{M^{2}_{kj}}{a_{k}^{2}}+\frac{M_{ij}^{2}}{a_{i}^{2}}\right]\times\\
\times g(v)\langle v\rangle^{\min(\g,0)}\d v, \end{multline}
where, for $\ell=1,2,3$,
\begin{equation}\label{newimp}
A_{\ell,\g}:=\begin{cases}\ds\inf_{\sigma \in \S^{1}}\int_{\R^{3}}\left(\sigma_{1}v_{\ell}-\sigma_{2}\right)^{2}g(v)\langle v\rangle^{\g}\d v \qquad &\text{ if } \g <0,\\
\frac{1}{3}a_{\ell} \qquad &\text{ if } \g \geq 0.
\end{cases}
\end{equation}
\end{lem}
\begin{proof}
Using \eqref{eq5}  for $i,j,k$ all distinct, we see that
\begin{equation} \label{eq7}
 \bigg( \frac1{a_i} - \frac1{a_k} \bigg) \,  ( K\,v_j - L_j )= \frac{M_{ij}}{a_i} -   \frac{M_{kj}}{a_k}  .
\end{equation}
Taking the square of \eqref{eq7} and integrating against $g(v)\langle v\rangle^{\min(\g,0)}$, we get
\begin{multline}\label{eqAa}
\left|\frac{1}{a_{i}}-\frac{1}{a_{k}}\right|^{2}\int_{\R^{3}}\left(Kv_{j}-L_{j}\right)^{2}g(v)\langle v\rangle^{\min(\g,0)}\d v\\
\leq \frac{2}{a_{i}^{2}}\int_{\R^{3}}M^{2}_{ij}g(v)\langle v\rangle^{\min(\g,0)}\d v + \frac{2}{a_{k}^{2}}\int_{\R^{3}}M_{kj}^{2}g(v)\langle v\rangle^{\min(\g,0)}\d v.\end{multline}
One observes then easily, by a homogeneity argument, that if $\g <0$, then
$$
\int_{\R^{3}}\left(Kv_{j}-L_{j}\right)^{2}g(v)\langle v\rangle^{\g}\d v \geq \left(K^{2}+L_{j}^{2}\right)\inf_{\sigma \in \S^{1}}\int_{\R^{3}}\left(\sigma_{1}v_{j}-\sigma_{2}\right)^{2}g(v)\langle v\rangle^{\g}\d v , $$
whereas, using \eqref{eq0}, we have 
$$\int_{\R^{3}}\left(Kv_{j}-L_{j}\right)^{2}g(v)\d v=(a_{j}\,K^{2}+L_{j}^{2}) \geq \frac{1}{3}(K^{2}+L_{j}^{2})\,a_{j}, $$
since $0\leq a_{j} \leq 3$. This shows that
$$\int_{\R^{3}}\left(Kv_{j}-L_{j}\right)^{2}g(v)\langle v\rangle^{\min(\g,0)}\d v \geq (K^{2}+L_{j}^{2})\,A_{j,\g}\,, \qquad \forall j=1,2,3, \quad \g \in \R,$$
so that \eqref{eqAa} reads
\begin{multline}\label{eq9}
\frac{A_{j,\g}}{2}\left|\frac{1}{a_{i}}-\frac{1}{a_{k}}\right|^{2}\left(K^{2}+L_{j}^{2}\right) 
\leq \frac{1}{a_{i}^{2}}\int_{\R^{3}}M^{2}_{ij}g(v)\langle v\rangle^{\min(\g,0)}\d v\\
 + \frac{1}{a_{k}^{2}}\int_{\R^{3}}M_{kj}^{2}g(v)\langle v\rangle^{\min(\g,0)}\d v.\end{multline}
Remembering that $a_1+a_2+a_3=3$, which can be rewritten (for $i,j,k$ all distinct)
\begin{equation} \label{eq10}
  1 - \frac1{a_i} =  \frac{a_j}3 \,\bigg(  \frac1{a_j} - \frac1{a_i} \bigg)  +  \frac{a_k}3 \,\bigg(  \frac1{a_k} - \frac1{a_i} \bigg) , 
\end{equation}
we  get first that
$$K^{2}\left|1-\frac{1}{a_{i}}\right|^{2} \leq \frac{2}{9}\left(a_{j}^{2}K^{2}\left|\frac{1}{a_{j}}-\frac{1}{a_{i}}\right|^{2}+a_{k}^{2}K^{2}\left|\frac{1}{a_{k}}-\frac{1}{a_{i}}\right|^{2}\right) , $$
whereas \eqref{eq9} implies
\begin{equation}\label{eq:9b}
a_{j}^{2}K^{2}\left|\frac{1}{a_{j}}-\frac{1}{a_{i}}\right|^{2} \leq 2\frac{a_{j}^{2}}{A_{k,\g}}\int_{\R^{3}}\left[\frac{M_{ik}^{2}}{a_{i}^{2}}+\frac{M_{jk}^{2}}{a_{j}^{2}}\right]g(v)\langle v\rangle^{\min(\g,0)}\d v , \end{equation}
and
\begin{equation*}
a_{k}^{2}K^{2}\left|\frac{1}{a_{k}}-\frac{1}{a_{i}}\right|^{2}\leq 2\frac{a_{k}^{2}}{A_{j,\g}}\int_{\R^{3}}\left[\frac{M^{2}_{ij}}{a_{i}^{2}}+\frac{M^{2}_{kj}}{a_{k}^{2}}\right]g(v)\langle v\rangle^{\min(\g,0)}\d v , \end{equation*}
which result in \eqref{eq11}.\end{proof}
 
We also get the following lemma:

\begin{lem}\label{lem:ABij} For any $\gamma \in \R$, we set $\bm{A}_{\g}=\max\ds\left(1,\frac{3}{\min_{\ell}A_{\ell,\g}}\right)$ and
\begin{equation}\label{defbij}
B_{ij}^{-1}=\inf_{\sigma \in \S^{1}}\int_{\R^{3}}\left|\sigma_{1} v_{i} -\sigma_{2} v_{j} \right|^{2}g(v)\langle v\rangle^{-2+\min(\g,0)}\d v \qquad i \neq j\,.
\end{equation}
One has, for all distinct $i,j,k \in \{1,2,3\}$
\begin{multline*}
 \frac{L_{j}^{2}}{a_{i}^{2}}+\frac{L_{i}^{2}}{a_{j}^{2}}  \leq 4B_{ij}\left(\int_{\R^{3}}N_{ij}^{2}g(v)\langle v\rangle^{\min(\g,0)}\d v \right.\\
 \left.+ 2\bm{A}_{\g}\int_{\R^{3}}\left[\frac{M^{2}_{ik}}{a_{i}^{2}}+\frac{M_{jk}^{2}}{a_{j}^{2}}+\frac{M^{2}_{ji}}{a_{j}^{2}}+\frac{M_{ij}^{2}}{a_{i}^{2}}\right]g(v)\langle v\rangle^{\min(\g,0)}\d v\right) . \end{multline*}
\end{lem}
\begin{proof} One observes that, inserting \eqref{eq5} in \eqref{eq4}, we get 
$$\frac{L_{i}}{a_{j}}v_{j}-\frac{L_{j}}{a_{i}}v_{i}=N_{ij}+\frac{M_{ij}}{a_{i}}v_{i} - \frac{M_{ji}}{a_{j}}v_{j} +\left(\frac{1}{a_{j}}-\frac{1}{a_{i}}\right)Kv_{i}v_{j} .$$
Integrating the square of this identity 
against $g(v)\langle v\rangle^{\min(\g,0)-2}$, we obtain
\begin{multline*}
\frac{1}{4}\int_{\R^{3}}g(v)\left|\frac{L_{i}}{a_{j}} v_{j} -\frac{L_{j}}{a_{i}} v_{i}\right|^{2}\langle v\rangle^{-2+\min(\g,0)}\d v\\
\leq \int_{\R^{3}}N_{ij}^{2}g(v)\langle v\rangle^{\min(\g-2,-2)}\d v+\int_{\R^{3}}\left(\frac{M_{ij}^{2}}{a_{i}^{2}}v_{i}^{2}+\frac{M_{ji}^{2}}{a_{j}^{2}}v_{j}^{2}\right)g(v)\langle v\rangle^{\min(\g-2,-2)}\d v \\
+\left(\frac{1}{a_{i}}-\frac{1}{a_{j}}\right)^{2}K^{2}\int_{\R^{3}}g(v)v_{i}^{2}v_{j}^{2}\langle v\rangle^{\min(\g-2,-2)}\d v.\end{multline*}
Clearly, by a homogeneity argument, the left-hand-side is bigger than 
$$\frac{1}{4}\left( \frac{L_{j}^{2}}{a_{i}^{2}}+\frac{L_{i}^{2}}{a_{j}^{2}}\right)B_{ij}^{-1}, $$
whereas, using the bound $v_{i}^{2}\langle v\rangle^{\min(\g-2,-2)} \leq 1$ and the fact that
$$\int_{\R^{3}}g(v)v_{i}^{2}v_{j}^{2}\langle v\rangle^{\min(\g-2,-2)}\d v \leq \int_{\R^{3}}g(v)|v|^{2}\d v=3, $$
we obtain the bound
\begin{multline*}
\frac{1}{4}\left( \frac{L_{j}^{2}}{a_{i}^{2}}+\frac{L_{i}^{2}}{a_{j}^{2}}\right)B_{ij}^{-1} \leq \int_{\R^{3}}N_{ij}^{2}g(v)\langle v\rangle^{\min(\g-2,-2)}\d v 
\\
+\int_{\R^{3}}\left(\frac{M_{ij}^{2}}{a_{i}^{2}} +\frac{M_{ji}^{2}}{a_{j}^{2}} \right)g(v) \langle v\rangle^{\min(\g,0)}\d v 
+3K^{2}\left(\frac{1}{a_{i}}-\frac{1}{a_{j}}\right)^{2}\\
 \leq \int_{\R^{3}}N_{ij}^{2}g(v)\langle v\rangle^{\min(\g,0)}\d v 
+\int_{\R^{3}}\left(\frac{M_{ij}^{2}}{a_{i}^{2}} +\frac{M_{ji}^{2}}{a_{j}^{2}} \right)g(v)\langle v\rangle^{\min(\g,0)} \d v \\+ \frac{6}{A_{k,\g}}\int_{\R^{3}}\left[\frac{M_{ik}^{2}}{a_{i}^{2}}+\frac{M_{jk}^{2}}{a_{j}^{2}}\right]g(v)\langle v\rangle^{\min(\g,0)}\d v ,
\end{multline*}
where we used \eqref{eq:9b} to estimate the last term. We obtain the desired result using the definition of $\bm{A}_{\g}.$
\end{proof}  
\begin{rmq}\label{rem:Bij1} Arguing exactly as  in \cite[Proof of Proposition 5, Estimate of $S_{f}$, p. 395-396]{CDH}, one can prove easily that, when 
\begin{equation}\label{eq:barH}
\int_{\R^{3}}g(v)\left|\log g(v)\right|\d v \leq \bar{H}, \qquad \quad \gamma <0,\end{equation}
there exist $c_{0} >0$ depending on $\bar{H}$ and  $c_{1} >0$ such that 
$$B^{-1}_{ij} \geq c_{0}\exp\left(-c_{1}\bar{H}\right)\,.$$
In the same way, for $\g <0$, and any $\sigma \in \S^{1}$, 
$$\int_{\R^{3}}g(v)|\sigma_{1}v_{\ell}-\sigma_{2}|^{2}\langle v\rangle^{\g}\d v \geq \langle R\rangle^{\g}\int_{C_{R}}g(v)\d v, \qquad \qquad R >0 , $$
where $C_{R}:=\{v \in \R^{3}\,,\,|v| \leq R\,,\,|\sigma_{1}v_{\ell}-\sigma_{2}| \geq 1\}.$ Then, 
\begin{multline*}
A_{\ell,\g} \geq \langle R\rangle^{\g}\left(1-\int_{|v| >R}g(v)\d v - \sup_{\sigma \in \S^{1}}\int_{|\sigma_{1}v_{\ell}-\sigma_{2}| < 1,\,|v| \leq R}g(v)\d v\right)\\
\geq \langle R\rangle^{\g}\left(\frac{1}{2}-\int_{|v| \leq R}g(v)\d v\right), \qquad \forall R >\sqrt{6},\end{multline*}
since $\int_{|v| >R}g(v)\d v \leq R^{-2}\int_{|v| >R}g(v)|v|^{2}\d v \leq \tfrac{3}{R^{2}}.$ Now, arguing exactly as in \cite[Proposition A.1]{AL-jmaa}, under \eqref{eq:barH}, there exists $C(\bar{H}) >0$ depending only on $\bar{H}$ such that
$$A_{\ell,\g} \geq C(\bar{H}).$$
\end{rmq}
\begin{rmq}\label{rem:Bij2}
Recall (see \cite[Proof of Lemma 2.4]{ABL}) that, if $g \in \mathcal{Y}_{\dd}(f_{\mathrm{in}})$ and satisfy the normalization conditions \eqref{eq0}--\eqref{eq:Mass} then
$$H(g) \leq H(f_{\mathrm{in}})+\varrho_{\mathrm{in}}{=H(f_{\mathrm{in}})+1} ,$$
where we recall that $\, H(g)=\int_{\R^{3}}g(v)\log g(v)\d v$.
Reminding also from \cite[Lemma A.1]{AL-jmaa} that there is some constant $c_{2}  >0$  such that
$$\int_{\R^{3}}g(v)|\log g(v)|\d v \leq H(g) + c_{2}{E_{\mathrm{in}}^{\frac{3}{5}} =H(g)+c_{2}}, $$
we see that, for $g \in \mathcal{Y}_{\dd}(f_{\mathrm{in}})$, \eqref{eq:barH} is met with $\bar{H}$ depending only on $H(f_{\mathrm{in}})$.\end{rmq}

We have all in hands to prove the functional  inequality.

\begin{prop}\label{theo:FisD}
Let $g$ satisfy \eqref{eq0}--\eqref{eq:Mass} and \eqref{infim}, and $\g \in \R$. Then, the following estimate holds: 
 {\begin{multline}\label{eq:entrop}
\bm{F}^{(\g)}(g):=\int_{\R^{3}}\left|\frac{\nabla g(v)}{g(v)(1-\dd g(v))}-Kv\right|^{2}g(v)\langle v\rangle^{\gamma}\d v \\
\leq  {170}\,\frac{\bm{e}_{\g}^{2}\,\bm{A}_{\g}\,}{\kappa_{0}^{2}}\,\max(1,\bm{B}_{\g})\,\max\left(1,\lm_{2+\gamma}(g)\right)\mathscr{I}_{\gamma}(g)\,\mathscr{D}_{\dd}(g),
\end{multline}}
where $\bm{B}_{\g}=\max_{i\neq j}B_{ij}$, $\bm{A}_{\g}=\max\left(1,\frac{3}{\min_{j}A_{j,\g}}\right)$, $\bm{e}_{\g}=\frac{3}{\min_{i}a_{i}}$, 
$$\mathscr{I}_{\gamma}(g)=\sup_{v \in \R^{3}}\langle v\rangle^{\gamma}\int_{\R^{3}}g(w)|w-v|^{-\gamma}\langle w\rangle^{2}\d w, $$
and we recall that $K=\dfrac{1}{\dd}\displaystyle\int_{\R^{3}}\log(1-\dd g)\d v$,
where $A_{j,\gamma}$ is defined in (\ref{newimp}), $a_i$ is defined in (\ref{eq6}), and $B_{ij}$ is defined in (\ref{defbij}).
\end{prop}

\begin{proof} We fix $j \in \{1,2,3\}$ and estimate the terms in \eqref{eq:fishj} using the previous Lemmas. For a fixed $i\neq j$, we denote for simplicity $S_{1},S_{2},S_{3}$ the first, second and third term on the right-hand-side of \eqref{eq:fishj}.  One deduces first from \eqref{eq11} that 
$$S_{1} \leq \frac{4\bm{A}_{\g}}{3}\,\lm_{2+\gamma}(g)\int_{\R^{3}}\left[\frac{M^{2}_{ik}}{a_{i}^{2}}+\frac{M_{jk}^{2}}{a_{j}^{2}}+\frac{M^{2}_{kj}}{a_{k}^{2}}+\frac{M_{ij}^{2}}{a_{i}^{2}}\right]g(v)\langle v\rangle^{\min(\g,0)}\d v, $$
where we used that $a_{k}^{2} \leq 9$ and $\frac{1}{A_{k,\g}} \leq \frac{1}{3}\bm{A}_{\g}$ for any $k \in \{1,2,3\}.$ According to Lemma \ref{lem:estimMN}, one obtains then
$$S_{1} \leq \frac{ {64}\,\bm{e}_{\g}^{2}\,\bm{A}_{\g}}{27\kappa_{0}^{2}}\lm_{2+\gamma}(g) \bm{I}_{\g}^{(2)}(g)\mathscr{D}_{\dd}(g).$$
Moreover, still using Lemma \ref{lem:estimMN},
$$S_{2} \leq \frac{4 \bm{e}_{\g}^{2}}{9\kappa_{0}^{2}}\,\bm{I}_{\g}^{(2)}(g)\mathscr{D}_{\dd}(g).$$
One also has from Lemma \ref{lem:ABij} 
$$S_{3} \leq 4B_{ij}\,\lm_{\gamma}(g)\left(\int_{\R^{3}}N_{ij}^{2}g(v)\langle v\rangle^{\min(\g,0)}\d v + \frac{8\bm{e}_{\g}^{2}\,\bm{A}_{\g}}{9}\max_{i,k}\int_{\R^{3}}M_{ik}^{2}g(v)\langle v\rangle^{\min(\g,0)}\d v\right), $$
which, from Lemma \ref{lem:estimMN} gives
$$S_{3} \leq \frac{16}{\kappa_{0}^{2}}B_{ij}\,\lm_{\gamma}(g)\left(\bm{I}_{\g}^{(0)}(g)+\frac{8\bm{e}_{\g}^{2}\,\bm{A}_{\g}}{9}\bm{I}_{\g}^{(2)}(g)\right)\mathscr{D}_{\dd}(g)\,.$$
Combining these estimates, summing up over $j=1,2,3$ and using that  $1 \leq \bm{A}_{\g},$ $1 \leq \bm{e}_{\g}$, while $\bm{I}_{\g}^{(0)}(g)+\bm{I}_{\g}^{(2)}(g)=\mathscr{I}_{\gamma}(g)$, we get
$$ \int_{\R^{3}}\big|\nabla h(v)-Kv\big|^{2} g(v)\langle v\rangle^{\gamma}\d v \le \sum_{j=1}^3 \int_{\R^{3}}\big|\partial_{j}h(v)-Kv_{j}\big|^{2} g(v)\langle v\rangle^{\gamma}\d v \le 9\,( S_1 + S_2 + S_3 ) $$
$$ \le9\, \frac{\mathscr{D}_{\dd}(g)}{\kappa_0^2}  \,\mathscr{I}_{\gamma}(g)\, \bm{e}_{\g}^2 \, \bigg\{ \frac{{64}}{27} \,\bm{A}_{\g}\, \lm_{2+\gamma}(g) + \frac49 + 16\, \max_{i\neq j} B_{ij}\, \lm_{\gamma}(g)  \max( \bm{e}_{\g}^{-2}, \frac89\, \bm{A}_{\g})\, \bigg\} $$
$$ \le 9\, \bm{A}_{\g}\, \frac{\bm{e}_{\g}^2 }{\kappa_0^2} \,\max(1, \bm{B}_{\g})\, \max(1, \lm_{2+\gamma})\, \mathscr{I}_{\gamma}(g)\, 
\, \bigg[\frac{{64}}{27} + \frac49 + 16 \bigg]\,\mathscr{D}_{\dd}(g) $$
which yields the desired estimate.
\end{proof} 
 
\section{Link with the relative entropy}\label{sec:relative}

The results of the previous section presented several functional inequalities for the entropy production $\mathscr{D}_{\dd}$ and some weighted Fisher information.   Since applications are oriented more to the study of the long-time behaviour of the solution to \eqref{LFDeq}, it is more relevant to link the entropy production to the \emph{Fermi-Dirac relative entropy} rather than the Fisher information. 
Proposition \ref{theo:FisD} is a first step in this direction since the entropy production $\mathscr{D}_{\dd}(g)$ controls some quantity $\bm{F}^{(\g)}(g)$ (see \eqref{eq:entrop}), which has to be interpreted as a weighted relative Fisher information for $h=\log g-\log(1-\dd g)$. To be able to get tractable estimates for the evolution of the relative entropy $\mathcal{H}_{\dd}(g\,|\M_{\dd})$, a comparison of such relative Fisher information to $\mathcal{H}_{\dd}(g\,|\M_{\dd})$ is needed.  We face two difficulties:
\begin{enumerate}[i)]
\item 
First, because of the parameter $K$, the Fermi-Dirac statistics is not making the weigthed Fisher information vanish. The relative Fisher information which can be related to the Fermi-Dirac entropy is, for function $g$ satisfying \eqref{eq0},
\begin{equation}\label{eq:Fe}
\mathscr{F}_{\dd}^{(\g)}(g)=\int_{\R^{3}}\left|\nabla h(v)+2b_{\dd}v\right|^{2}g(v)\langle v\rangle^{\min(\g,0)}\d v,\end{equation}
since it is the one which vanishes for $g=\M_{\dd}$ given by \eqref{eq:FDS}. 
\item Second, we need an additional general functional inequality which allows to {link the}
 relative Fisher information to the relative entropy. In the classical framework, such a link is well-known and is given by \emph{the logarithmic Sobolev inequality} \cite[Chapter 2]{jungel}.
\end{enumerate}

\subsection{The hard potential  (and Maxwell molecules)  case} In this paragraph, we assume $\g \geq 0.$ We still consider here functions $g$ in the class $\mathcal{Y}_{\dd}(f_{\mathrm{in}})$ as in Section \ref{sec:entLT}. In this case, 
the weighted relative Fisher information $\bm{F}^{(\g)}$ in \eqref{eq:entrop} can be bounded from below by $\bm{F}^{(0)}$, i.e.
$$\bm{F}^{(\g)}(g) \geq \bm{F}^{(0)}(g)=\int_{\R^{3}}\left|\nabla h(v)-Kv\right|^{2}g(v)\d v , $$
and inequality \eqref{eq:entrop} reads
\begin{equation}\label{eq:entroFK}
\int_{\R^{3}}\left|\nabla h(v)-Kv\right|^{2}g(v)\langle v\rangle^{\gamma}\d v  \leq \bm{F}^{(\g)}(g) \leq \lambda^{-1}(g)\mathscr{D}_{\dd}(g) ,
\end{equation}
where $\lambda^{-1}(g)$ is given by \eqref{eq:lambdag}, since  {$A_{\gamma} = 3 \,\bm{e}_{\g}$ when $\gamma \ge 0$}. Now,
$$\bm{F}^{(0)}(g)=\int_{\R^{3}}|\nabla h|^{2}g(v)\d v -2K\int_{\R^{3}}\left(g\nabla h\right) \cdot v \d v + K^{2}\int_{\R^{3}}g|v|^{2}\d v , $$
where we recall that $h$ is given by \eqref{newh}. Using this with the identity $g \nabla h = - \frac1{\dd} \nabla \log (1 - \dd g)$ and the fact that the energy of $g$ is $3$, gives
$$\int_{\R^{3}} \left|\nabla h(v)-Kv\right|^{2} {g(v)}\d v=\int_{\R^{3}}|\nabla h|^{2}g(v)\d v -3K^{2}.$$
With the notation \eqref{eq:Fe}, the same considerations give
$$\mathscr{F}_{\dd}^{(0)}(g)=\int_{\R^{3}}|\nabla h|^{2}g(v)\d v + 12b_{\dd}^{2}+4b_{\dd}\int_{\R^{3}}g\nabla h \cdot v\d v$$
$$ =\int_{\R^{3}}|\nabla h|^{2}g(v)\d v +12b_{\dd}(b_{\dd}+K) . $$
Taking this into account, inequality \eqref{eq:entroFK} reads  
$$\lambda(g)\left(\int_{\R^{3}}|\nabla h|^{2}g(v)\d v -3K^{2}\right) \leq \mathscr{D}_{\dd}(g),$$ or equivalently
\begin{equation}\label{eq:lam}
\lambda(g) {\mathscr{F}^{(0)}_{\dd}(g)} -3\lambda(g)\,\left(K+2b_{\dd}\right)^{2} \leq \mathscr{D}_{\dd}(g).\end{equation}
It is known from general results in \cite{To} that $\mathscr{F}^{(0)}_{\dd}(g)$ controls the relative entropy $\mathcal{H}_{\dd}(g|\M_{\dd})$ (see \cite[Eq. (3.21)]{CLR}). Precisely,
\begin{equation}\label{eq:LogSob}
{\mathscr{F}^{(0)}_{\dd}(g)} \geq 2b_{\dd}\mathcal{H}_{\dd}(g|\M_{\dd}).\end{equation}
We therefore need to compare $(K+2b_{\dd})$ with $\mathscr{D}_{\dd}(g)$. 
\begin{lem}
If $g \in \mathcal{Y}_{\dd}(f_{\rm in})$ satisfies \eqref{eq0}--\eqref{eq:Mass} and \eqref{infim} and $\M_{\dd}$ is the associated Fermi-Dirac statistics, 
$$\left|K+2b_{\dd}\right| \leq \frac{2\dd}{\kappa_{0}^{2}}\max(\|g\|_{\infty},\|\M_{\dd}\|_{\infty})\|g-\M_{\dd}\|_{L^{1}}.$$\end{lem}
\begin{proof} One has 
\begin{equation*}\begin{split}
K+2b_{\dd}&=\frac{1}{\dd}\int_{\R^{3}}\left[\log(1-\dd\,g)-\log(1-\dd\M_{\dd})+\dd(g-\M_{\dd})\right]\d v
\\&=\frac{1}{\dd}\int_{\R^{3}}\left[\zeta(\dd\,g(v))-\zeta(\dd\M_{\dd}(v))\right]\d v , 
\end{split}\end{equation*}
where we introduced the function $\zeta(x)=\log(1-x)+x$, $x >0$. Observe  that
$$\zeta(x)=-\int_{0}^{x}\frac{x-t}{(1-t)^{2}}\d t, \qquad \forall x >0, $$
so that
$$\zeta(\dd x)-\zeta(\dd y)=\dd^{2}\left(-\int_{0}^{x}\frac{x-t}{(1-\dd\,t)^{2}}\d t + \int_{0}^{y}\frac{y-t}{(1-\dd t)^{2}}\d t\right).$$
Assuming that $\min(1-\dd x,1-\dd y) \geq \kappa_{0}$ and, say, $x >y$, one gets
\begin{equation*}\begin{split}
\frac{1}{\dd}\left|\zeta(\dd x)-\zeta(\dd y)\right|&=\dd\left|\int_{0}^{y}\frac{x-y}{(1-\dd t)^{2}}\d t + \int_{y}^{x}\frac{x-t}{(1-\dd t)^{2}}\d t\right|\\
&\leq \dd\frac{x+y}{\kappa_{0}^{2}}|x-y| , \end{split}\end{equation*}
so that
$$|K+2b_{\dd}| \leq \frac{\dd}{\kappa_{0}^{2}}\int_{\R^{3}}\left|g-\M_{\dd}\right|\,(g+\M_{\dd})\d v , $$
which gives the result.\end{proof}

We deduce from the above the full proof of our main result:
\begin{proof}[Proof of Theorem \ref{theo:main-entro}] Combining \eqref{eq:lam} and \eqref{eq:LogSob} with the previous Lemma, we get
$$\mathscr{D}_{\dd}(g)\geq 2\lambda(g)\,b_{\dd}\mathcal{H}_{\dd}(g|\M_{\dd}) -\frac{12}{\kappa_{0}^{4}}\dd^{2}\lambda(g)\|g-\M_{\dd}\|_{L^{1}}^{2}\max(\|g\|_{\infty}^{2},\|\M_{\dd}\|_{\infty}^{2}) , $$
which gives the result thanks to Czizar-Kullback inequality \eqref{eq:csiszar}. \end{proof}
 
A crucial point for the long time behaviour of solutions to Landau-Fermi-Dirac equation is that $\|g\|_{\infty}$ is independent of $\dd$ (in the large time), so that we can take $\dd$ sufficiently small for the coefficient in the theorem to be positive. 

\subsection{The soft-potential case} The case $\gamma < 0$ is difficult due to the weight $\langle \cdot \rangle^{\g}$ in $\bm{F}^{(\g)}(g)$. In particular, no functional inequality in the spirit of \eqref{eq:LogSob} seems available in this case. In order to exploit the entropy production $\mathscr{D}_{\dd}(g)$ to deduce results about the long time behaviour of solutions to \eqref{LFDeq}, a possible route is to adapt the strategy initiated in \cite{TV} which consists in looking for an interpolation estimate between the entropy dissipation associated to $\gamma < 0$ and $\gamma \geq 0$.  More specifically, denoting the entropy production associated to the potential $\Psi(z)=|z|^{s+2}$ by $\mathscr{D}_{\dd}^{(s)}(g)$,
$$\mathscr{D}_{\dd}^{(s)}(g)=\frac{1}{2}\int_{\R^{3}\times\R^{3}}|v-\vet|^{s+2}\bm{\Xi}_{\dd}[g](v,\vet)\d v\d\vet , $$
with $\bm{\Xi}_{\dd}[g]$ defined by \eqref{eq:Xidd}, a use of H\"older's inequality implies that, for $0 \leq g\leq \dd^{-1}$ with $\mathcal{S}_{\dd}(g) < \infty$,
$$\mathscr{D}_{\dd}^{(0)}(g) \leq \left(\mathscr{D}_{\dd}^{(\g)}(g)\right)^{\frac{s}{s-\g}}\,\left(\mathscr{D}_{\dd}^{(s)}(g)\right)^{-\frac{\g}{s-\g}}, \qquad \forall\, \gamma < 0 \leq s\,.$$
Using Theorem \ref{theo:main-entro} for $\g=0$, one sees that
$$\mathscr{D}^{(0)}_{\dd}(g)\geq \,\bar{\lambda}_{\dd}(g)\,\mathcal{H}_{\dd}(g|\M_{\dd})\,,$$
for some function $\bar{\lambda}_{\dd}(g)$ which can be chosen to be positive under a smallness assumption on $\dd >0$ provided that $\|g\|_{\infty} \leq C < \dd^{-1}$ for some positive constant independent of $\dd$.

\noindent 
Assuming such inequality, we deduce that
$$\mathscr{D}^{(\g)}_{\dd}(g) \geq \bar{\lambda}_{\dd}(g)^{1-\frac{\g}{s}}\,  {\left(\mathscr{D}_{\dd}^{(s)}(g)\right)^{-\frac{|\g|}{s}} }\,\mathcal{H}_{\dd}(g|\M_{\dd})^{1-\frac{\g}{s}}, \qquad \g < 0 < s.$$
In the application to the equation, it is possible to prove that $C$ depends only on mass, energy, and initial entropy of $g$.  Thus, there exists an explicit $\dd_{\star} >0$ and $\lambda_{0} >0$ such that
 {$$\inf\{\bar{\lambda}_{\dd}(g)\;;\;\|g\|_{\infty} \leq C\} \geq \lambda_{0} >0, \qquad \forall \dd \in (0,\dd_{\star}].$$}
Therefore, to bound from below $\mathscr{D}_{\dd}^{(\g)}(g)$ by some power of $\mathcal{H}_{\dd}(g|\M_{\dd})$, it suffices to provide an upper bound for $\mathscr{D}^{(s)}_{\dd}(g)$, hopefully independent of $\dd$.

\smallskip
\noindent
It is not difficult to check that
 {\begin{equation}\label{Ds:est}
\mathscr{D}^{(s)}_{\dd}(g) \leq 8\,\frac{2^{s/2}}{\kappa_{0}}\,\lm_{s+2}(g)\int_{\R^{3}}\langle v \rangle^{s+2}\big| \nabla \sqrt{g}\big|^{2}\d v\,,\qquad s \geq0.
\end{equation}}
Therefore, establishing an upper bound for $\mathscr{D}_{\dd}^{(s)}(g)$ is equivalent to establish an  upper bound for the weighted Fisher information $\int_{\R^{3}}\langle v \rangle^{s+2}\big| \nabla \sqrt{g}\big|^{2}\d v$.\\

\noindent
This approach is adopted successfully for the solution to \eqref{LFDeq} in our contribution \cite{ABDL-soft} for moderate soft potentials $\gamma \in (-2,0)$.  We refer to \cite[Section 6]{ABDL-soft} for more details, but we can already point out that several major technical difficulties arise when trying to follow this approach: 
\begin{enumerate}
\item First, one needs to obtain a control of the various terms in Theorem \ref{theo:main-entro}, i.e. regularity estimates, moment estimates and related issues. All these quantities depend of course on the time $t$,
but  the \emph{a priori} estimates turn out to  
 grow moderately with time.
\item Second and more difficult, one needs to show that the solution $f(t,v)$ to \eqref{LFDeq} satisfies the lower bound \eqref{infim}.
Such bounds are derived thanks to a new approach based on a suitable level set approach inspired by De Giorgi's method for parabolic equation and introduced in \cite{ricardo} for the study of Boltzmann equation.
\item Finally, because of the upper bound \eqref{Ds:est}, some upper bound for the Fisher information with weights has to be established for solution to \eqref{LFDeq}. This can be done by adapting to the Landau-Fermi-Dirac case the recent approach of \cite{fisher}, introduced for both the Boltzmann and Landau equation.
\end{enumerate}

\section{Application to the long time behaviour of solutions for hard potentials}\label{sec:LT}

In all this Section, we consider hard potentials. We aim here to complement the results obtained in \cite{ABL} regarding the long time behaviour of the solutions to \eqref{LFDeq} in that case. We recall that in reference \cite{ABL} exponential convergence to equilibrium has been obtained for $\g\in (0,1]$  by combining a non-optimal functional inequality for $\mathscr{D}_{\dd}(f(t))$ together with a careful spectral analysis of the linearized equation.  Based on the results of the previous section, we prove that exponential convergence can be deduced \emph{directly} thanks to the entropy production appearing in Theorem \ref{theo:main-entro} which is an analogue of a proof of Cercignani's conjecture, refer to \cite{DMouVil}, for the Landau-Fermi-Dirac equation.

\subsection{Reminders of known estimates}

We recall that the following results were obtained in \cite[Theorem 3.1]{ABL}.
\begin{prop}\label{moments}
Consider {$0 < \g \leq 1$ and} $0\leq f_{\mathrm{in}}\in L^{1}_{s_{\gamma}}(\R^{3})$, with $s_{\gamma}=\max\big\{\tfrac{3\gamma}{2}+2, 4-\gamma\big\}$, satisfying \eqref{hypci} together with the normalization \eqref{eq0}--\eqref{eq:Mass}.  Let $f=f(t,v)$ be a weak solution to the LFD equation.

\medskip
\noindent
(i) Then, for any $s\geq0$
$$\int_{t_0}^{T} \int_{\R^3} |\nabla f(t,v)|^{2}\langle v\rangle^{s+\g} \, \d v \, \d t <+\infty\,,\qquad \forall\,T>t_0>0\,.$$

\medskip
\noindent
(ii) There exists some positive constant $C_{t_0}$ depending on {$H(f_{\mathrm{in}})$}, $s$ and $t_0$, but not on $\dd$, such that 
\begin{equation}\label{unif_L2n}
\int_{\R^3} \big( f(t,v)+  f^2(t,v) \big)\, \langle v\rangle^{s} \, \d v \leq C_{t_0}\,,\qquad \forall\,s\geq0\,,\;t\geq t_0>0\,.
\end{equation}
Moreover, if
\begin{equation*}
\int_{\R^3} \big( f(0,v)+  f^2(0,v) \big)\, \langle v\rangle^{s} \, \d v < \infty\,,
\end{equation*}
then $t_0=0$ is a valid choice in the estimate \eqref{unif_L2n} with a constant depending on such initial quantity.
\end{prop}

We also recall the following pointwise estimates, taken from \cite[Corollary 3.7]{ABL}.

\begin{prop}\label{cor:Linfty}
Consider $0\leq f_{\mathrm{in}}\in L^{1}_{s_{\gamma}}(\R^{3})$ satisfying \eqref{hypci} together with the normalization \eqref{eq0}--\eqref{eq:Mass} and $\g \in (0,1]$. Then, for any solution $f(t)=f_{\dd}(t)$ to \eqref{LFDeq}, it holds
$$\sup_{t \geq t_{0}}\left\|f(t)\right\|_{\infty} \leq \bar{C}_{t_{0}}\,,\qquad \forall\, t_{0}>0.$$
The constant $\bar{C}_{t_{0}}$ only depends on  $H(f_{\mathrm{in}})$ and $t_0$.

\smallskip
\noindent
Consequently, for any $t_{0} >0$ and $\bar{\kappa}_{0} \in (0,1)$, there exists $\dd_{\star}>0$ depending only on $\bar{\kappa}_{0}$, $t_{0}$  and  $H(f_{\mathrm{in}})$, such that
\begin{equation}\label{eq:lower*}
\inf_{v \in \R^{3}}\big(1-\dd\,f(t,v)\big) \geq \bar{\kappa}_{0}, \qquad \forall\, \dd \in (0,\dd_{\star}),\,   t \geq t_{0}.
\end{equation}\end{prop}

\subsection{Long time behaviour}

 We first observe that for any  $g \in \mathcal{Y}_{\dd}(f_{\mathrm{in}})$, the quantity $\lambda(g)$ is bounded below by a strictly positive constant depending only on $\bm{m}_{2+\gamma}(g)$ and $\|g\|_{L^{2}_{2+\g}}$. This is a consequence of the two following lemmas.
 \begin{lem} \label{nl1}
Let $\gamma \in [0,1]$. 
 For any  $g \in \mathcal{Y}_{\dd}(f_{\mathrm{in}}) \cap L^{2}_{2+\g}(\R^{3})$ satisfying the normalization conditions \eqref{eq0}--\eqref{eq:Mass}, it holds that
\begin{equation}\label{e00}
\mathscr{I}_{\gamma}(g)   \le  c_{\gamma} \, \left(1 + \bm{m}_{2+\gamma}(g) + \left\|g\right\|_{L^2_{2+\g}}\right) , 
\end{equation}
for some $c_\gamma>0$ depending only on $\gamma$.
\end{lem}

\begin{proof}
For all $v \in \R^3$, we split the integration in $\langle v\rangle^{\g}\int_{\R^{3}}g(w)|v-w|^{-\g}\langle w\rangle^{2}\d w$ according to the regions $\{|v-w| \leq \tfrac{|v|}{2}\}$ and $$
\{|v-w| \geq \tfrac{|v|}{2}\}=\{|v-w| \geq \tfrac{|v|}{2}\;;\;|v| \leq 1\} \cup \{|v-w| \geq \tfrac{|v|}{2}\;;\;|v| \geq 1\}$$ to get
\begin{multline*}
\langle v\rangle^{\g}\int_{\R^{3}}g(w)|v-w|^{-\g}\langle w\rangle^{2}\d w
\le \langle v\rangle^\gamma \int_{|v |\le 2 |w| }  g(w) |v-w|^{-\gamma} \langle w\rangle^2\, \d w\\
+2^{\frac{\gamma}{2}}  \int_{\R^{3}}  g(w) |v-w|^{-\gamma} \langle w\rangle^2\, \d w + \langle v\rangle^{\gamma}\int_{\{|v - w|\geq \tfrac{|v|}{2}\,;\, |v| \ge 1\}}  g(w) \left(\frac{|v|}{2}\right)^{- \gamma} \langle w\rangle^2\, \d w\\
\leq 2^{\g}\int_{\R^{3}}g(w)|v-w|^{-\g}\langle w\rangle^{2+\g}\d w+2^{\frac{\g}{2}}\int_{\R^{3}}g(w)|v-w|^{-\g}\langle w\rangle^{2}\d w \\
+ 2^{\frac{3}{2}\g}\underset{=4}{\underbrace{\int_{\R^{3}}g(w)\langle w\rangle^{2}\d w}}.\end{multline*}
Each of the first two integrals can be bounded by $\int_{\R^{3}}g(w)|v-w|^{-\g}\langle w\rangle^{2+\g}\d w$ since $\g \geq0.$ Therefore, using the rough estimate $2^{\g}+2^{\frac{\g}{2}} \leq 2^{\g+1}$, we end up with
\begin{multline*}
\langle v\rangle^{\g}\int_{\R^{3}}g(w)|v-w|^{-\g}\langle w\rangle^{2}\d w
 \le 2^{\frac{3}{2}\gamma+2} + 2^{\gamma +1}  \int_{|v - w|\ge 1}  g(w)\, |v-w|^{-\gamma}  \,\langle w\rangle^{2+ \gamma}\, \d w \\
+ 2^{\gamma + 1}  \int_{|v - w|\le 1}  g(w)\, |v-w|^{-\gamma}  \,\langle w\rangle^{2+ \gamma}\, \d w.\end{multline*}
As a consequence,
 {$$ \mathscr{I}_{\gamma}(g)  \le 2^{\gamma + 1}  \,(2^{\frac{\g}{2}+1}+ \bm{m}_{2+\gamma}(g)) + 2^{\gamma +1}  \, \left\|g \langle \cdot\rangle^{2+\gamma}
* (|\cdot|^{-\gamma} 1_{|\cdot| \le 1})\right\|_{L^{\infty}} $$}
and the result follows from Young's convolution inequality since $\int_{|v| \leq1}|v|^{-2\g}\d v  < \infty$ for $\g \in (0,\frac{3}{2})$.
\end{proof}
We now have the equivalent of the bounds in Remarks \ref{rem:Bij1} and \ref{rem:Bij2} for $\g \in [0,1]$.
\begin{lem} \label{nl2}
Let $\gamma \in [0,1]$. 
 For any  $g \in \mathcal{Y}_{\dd}(f_{\mathrm{in}}) \cap L^{2}(\R^{3})$, it holds that
\begin{equation}\label{e000}
\frac{1}{\bm{B}_{\g}}:=\min_{i\neq j}\inf_{\sigma \in \S^{1}}\int_{\R^{3}}\left|\sigma_{1}\frac{v_{i}}{\langle v\rangle}-\sigma_{2}\frac{v_{j}}{\langle v\rangle}\right|^{2}g(v)\d v, \qquad \frac{1}{\bm{e}_{\g}}=\min_{i}\frac{1}{3}\int_{\R^{3}} g(v)\,v_{i}^2\,\d v
\end{equation}
are bounded below by a strictly positive constant depending only on $\|g\|_{L^2}$.
\end{lem}

\begin{proof}
We observe first that 
$$  \frac{1}{\bm{e}_{\g}} \ge \min_{i} \frac13 \int_{\R^3} g(v)\, \left| \frac{v_{i}}{\langle v\rangle} \right|^2 \, \d v ,$$
so that it is sufficient to bound below $\frac{1}{\bm{B}_{\g}}$. 
\par
Then for all $\tau \ge 0$ and $\sigma \in \S^{1}$, with the notations of Lemma \ref{L2unif} \textit{(1)},
\begin{multline*}
\int_{\R^{3}}\left|\sigma_{1}\frac{v_{i}}{\langle v\rangle}-\sigma_{2}\frac{v_{j}}{\langle v\rangle}\right|^{2}g(v)\d v  \ge 
\tau^2  \int_{\R^{3}}  \ind_{\left\{  \left|\sigma_{1}\frac{v_{i}}{\langle v\rangle}-\sigma_{2}\frac{v_{j}}{\langle v\rangle}\right| \ge \tau \right\} } 
\, \, g(v)\d v  \\
 \ge  \tau^2 \left\{ \int_{|v| \le R(f_{\mathrm{in}})} g(v)\,\d v -  \int_{|v| \le R(f_{\mathrm{in}})}  \ind_{\left\{  \left|\sigma_{1}\frac{v_{i}}{\langle v\rangle}-\sigma_{2}\frac{v_{j}}{\langle v\rangle}\right| \leq \tau \right\} } \, \, g(v)\d v \right\} \\
   \ge  \tau^2 \left\{ \eta(f_{\mathrm{in}}) - \|g\|_{L^2} \, \left|B(0, R(f_{\mathrm{in}})) \cap \left\{ v \in \R^3, \,\, \left| \sigma_{1}\frac{v_{i}}{\langle v\rangle}-\sigma_{2}\frac{v_{j}}{\langle v\rangle} \right|  \le \tau \right\} \right|^{1/2}  \right\} .\end{multline*}
This last quantity is bounded below by a strictly positive constant, uniformly with respect to $\sigma$ and $i \neq j$, when $\tau$ is small enough.
\end{proof}
In that case, Theorem \ref{theo:main-entro} can be improved as following:
\begin{cor}\label{cor:entroLe} Assume that $\g \in [0,1]$ and   $g  \in \mathcal{Y}_{\dd}(f_{\mathrm{in}}) \cap L^{2}_{2+\g}(\R^{3})$ satisfies   the normalization  conditions \eqref{eq0}--\eqref{eq:Mass}.
Then, there exists a positive constant $C_{0}[g]>0$ depending only on $\kappa_{0}$ and some (upper bounds on) $\lm_{2+\g}(g)$ and  $\|g\|_{L^{2}_{2+\g}}$, such that
$$\mathscr{D}_{\dd}(g) \geq C_{0}[g]\left[b_{\dd}-\frac{12\dd^{2}}{\kappa_{0}^{4}}\max(\|g\|_{\infty}^{2},\|\M_{\dd}\|_{\infty}^{2})\right]
\mathcal{H}_{\dd}(g|\M_{\dd}).$$
\end{cor}
\begin{proof} The proof is a simple application of Theorem \ref{theo:main-entro} once we notice that, thanks to the two previous Lemmas, it is possible to bound $\lambda(g)$ from below by a positive constant depending on $\kappa_{0},\|g\|_{L^{2}_{2+\g}}$ and $\lm_{2+\g}(g)$. 
\end{proof}

We now are in a position to give a direct proof of the exponential decay of the relative entropy drastically improving \cite[Theorem 6.10]{ABL}.

\begin{proof}[Proof of Theorem \ref{cvex}] The proof is a simple combination of Corollary \ref{cor:entroLe}  together with Propositions \ref{moments} and \ref{cor:Linfty}. Namely, for $t_{0} >0$ and $\bar{\kappa}_{0} >0$,  the solution $f(t)$ to \eqref{LFDeq} satisfies the lower bound \eqref{eq:lower*} for any $t >t_{0}$ and any $\dd \in (0,\dd_{\star})$. Moreover, according to Proposition \ref{moments}, $\sup_{t \geq t_{0}}\lm_{2+\g}(f(t)) \leq C_{t_{0}}$ and $\sup_{t\geq t_{0}}\|f(t)\|_{L^{2}_{2+\g}} \leq C_{t_{0}}.$ Consequently, with the notations of Corollary \ref{cor:entroLe}, one has
$$\inf_{t\geq t_{0}}C_{0}[f(t)]=:\nu >0.$$
Now, using again Proposition \ref{cor:Linfty}, $\sup_{t\geq t_{0}}\|f(t)\|_{L^{\infty}} \leq \bar{C}_{t_{0}}$.  Moreover, according to \cite[Lemma A.1]{ABL}, we know that $b_{\dd} >\frac{1}{8}$  {for any $\dd \le \overline{\dd}=\left(\frac25\right)^{\frac52}(18\pi)^{\frac32}$} and there is a universal numerical constant $c_{0} >0$ such that $\sup_{\dd \in  {(0,\overline{\dd})}}\|\M_{\dd}\|_{\infty} \leq c_{0}$\footnote{Namely, $\|\M_{\dd}\|_{L^{\infty}} \leq a_{\dd} \leq \frac{5}{3}\left(\frac{5}{18\pi}\right)^{\frac{3}{2}}$ {for any $\dd\le \overline{\dd}$} according to \cite[Lemma A.1]{ABL}.} , one can choose  { $\dd^{\dagger}\in(0,\overline{\dd})$} such that, say,
$$\frac{1}{8}-\frac{12\dd^{2}}{\bar{\kappa}_{0}^{4}}\max\left(\bar{C}_{t_{0}}^{2},c_{0}^{2}\right) >\frac{1}{16}, \qquad \forall \dd \in (0,\dd^{\dagger}).$$
Setting now $\mu:=\frac{1}{16}\nu$, one deduces from Corollary \eqref{cor:entroLe} that
$$\mathscr{D}_{\dd}(f(t)) \geq \mu\,\mathcal{H}_{\dd}(f(t)|\M_{\dd}) \qquad \forall t \geq t_{0}.$$
Recall now that
$$\dfrac{\d}{\d t}\mathcal{H}_{\dd}(f(t)|\M_{\dd})=-\mathscr{D}_{\dd}(f(t)), \qquad \forall t \geq0,$$
so that we deduce after integration that
$$\mathcal{H}_{\dd}(f(t)|\M_{\dd}) \leq \mathcal{H}_{\dd}(f(t_{0})|\M_{\dd})\exp\left(-\mu(t-t_{0})\right), \qquad \forall t \geq t_{0}, $$
and the result follows since $\mathcal{H}_{\dd}(f_{t_{0}}|\M_{\dd}) \leq \mathcal{H}_{\dd}(f_{\mathrm{in}}|\M_{\dd}).$
\end{proof}

\end{document}